\documentclass[12pt]{amsart}

\usepackage[dvipsnames]{xcolor}
\colorlet{RED}{red}

\usepackage{enumerate}
\usepackage{enumitem}

\usepackage{graphicx}
\usepackage{amssymb,amsmath, amsthm} 
\usepackage{booktabs}
\usepackage{hyperref}
\usepackage{ mathrsfs }
\usepackage{algorithm}
\usepackage{fullpage}
\usepackage{colonequals}

\newtheorem{thm}{Theorem}[section]
\newtheorem{lemma}[thm]{Lemma} 
\newtheorem{prop}[thm]{Proposition}
\newtheorem*{prop*}{Proposition}
\newtheorem{comp}[thm]{Computation}

\theoremstyle{definition}
\newtheorem{defn}[thm]{Definition}
\newtheorem{question}[thm]{Question}

\newtheorem{example}[thm]{Example}

\newtheorem{rmk}[thm]{Remark}

\newcommand{\Z}{\mathbb Z}
\newcommand{\N}{\mathbb N}
\newcommand{\R}{\mathbb R}
\newcommand{\C}{\mathbb C}
\newcommand{\Q}{\mathbb Q}
\newcommand{\F}{\mathbb F}
\newcommand{\A}{\mathbb A}

\newcommand{\PP}{\mathbb P}

\newcommand{\mc}{\mathcal}

\newcommand{\red}{\mathrm{red}}

\newcommand{\ovl}{\overline} 
 
\DeclareMathOperator{\PGL}{PGL}
\DeclareMathOperator{\LS}{LS}
\DeclareMathOperator{\LSF}{LSF}
\DeclareMathOperator{\Spec}{Spec}

\DeclareMathOperator{\Fano}{Fano}
\DeclareMathOperator{\MK}{MK}
\DeclareMathOperator{\Pappus}{Pappus}

\DeclareMathOperator{\Star}{Star}

\newcommand{\GITQ}{/ \! /}
 
\usepackage{pgfplots}
\pgfplotsset{compat=1.15}
\usepackage{mathrsfs}
\usetikzlibrary{arrows}

\begin{document}

\title[Configurations of 10 points]{Configurations of 10 points and their incidence varieties}
\author[K. Isham, N. Kaplan, S. Kimport, R. Lawrence, L. Peilen, and M. Weinreich]{Kelly Isham, Nathan Kaplan, Sam Kimport, Rachel Lawrence, Luke Peilen, and Max Weinreich}
\date{\today}

\address{\parbox{\linewidth}{Department of Mathematics, Colgate University, Hamilton, NY 13346.\\
ORCID: 0000-0001-7979-0516\\}
}
\email{kisham@colgate.edu}

\address{\parbox{\linewidth}{Department of Mathematics, University of California, Irvine, CA 92697.\\
ORCID: 0000-0001-5305-983X \\}
}
\email{nckaplan@math.uci.edu}

\address{\parbox{\linewidth}{Corporate Research, Robert Bosch LLC, Sunnyvale, CA 94085\\}}
\email{skimport.math@gmail.com}

\address{\parbox{\linewidth}{Microsoft Research Cambridge, Cambridge, UK CB1 2FB.\\
ORCID: 0009-0002-7958-9657\\}
}
\email{rlaw@berkeley.edu}

\address{\parbox{\linewidth}{Department of Mathematics, Temple University, Philadelphia, PA 19122.\\
ORCID: 0000-0003-1400-0610\\}}
\email{luke.peilen@temple.edu}

\address{\parbox{\linewidth}{Department of Mathematics, Harvard University, Cambridge, MA 02138.\\
 ORCID: 0000-0002-0103-2245\\}
}
\email{mweinreich@math.harvard.edu}



\begin{abstract}
    Incidence varieties are spaces of $n$-tuples of points in the projective plane that satisfy a given set of collinearity conditions. We classify the components of incidence varieties and realization moduli spaces associated to configurations of up to $10$ points, up to birational equivalence. We show that each realization space component is birational to a projective space, a genus $1$ curve, or a K3 surface. To do this, we reduce the problem to a study of $163$ special arrangements called superfigurations. Then we use computer algebra to describe the realization space of each superfiguration.
\end{abstract}
\maketitle

\section{Introduction}

\begin{defn} \label{def_iv_intro}
    Fix a field $k$, and let $\PP^2 = \PP^2_k$ be the projective plane over $k$. Let $n \in \N$, and let $L$ be a collection of subsets of $\{1,2,\ldots, n\}$. The elements of $L$ are called \emph{lines} or \emph{collinearity conditions}. A \emph{weak realization} of $(n, L)$ over $k$ is an $n$-tuple of points $(p_1, \ldots, p_n)$ in $\PP^2$ such that, for each $\ell \in L$, the set of points $\{p_i : i \in \ell\}$ is collinear.
    
    The \emph{incidence variety} $\mc{I}(L) = \mc{I}(n,L)$ over $k$ is the subvariety of $(\PP^2)^n$ consisting of all weak realizations of $(n, L)$.
\end{defn}

Incidence varieties are parameter spaces for point configurations with prescribed lines. For example, the
\emph{Fano plane} (Figure \ref{fig_fano}) is the combinatorial data of the point set $\{1,2,\ldots,7\}$ with the collinearity conditions
$$L_{\Fano} = \{\{1,2,3\}, \{1,4,5\}, \{1, 6, 7\}, \{3,4,7\}, \{ 3,5,6\}, \{2,5,7\}, \{2,4,6\}\}.$$
The incidence variety $\mc{I}(7, L_{\Fano})$ parametrizes $7$-tuples of points in $\PP^2$ that form a (possibly degenerate) Fano plane. 

    \begin{figure}[h]
    \includegraphics[]{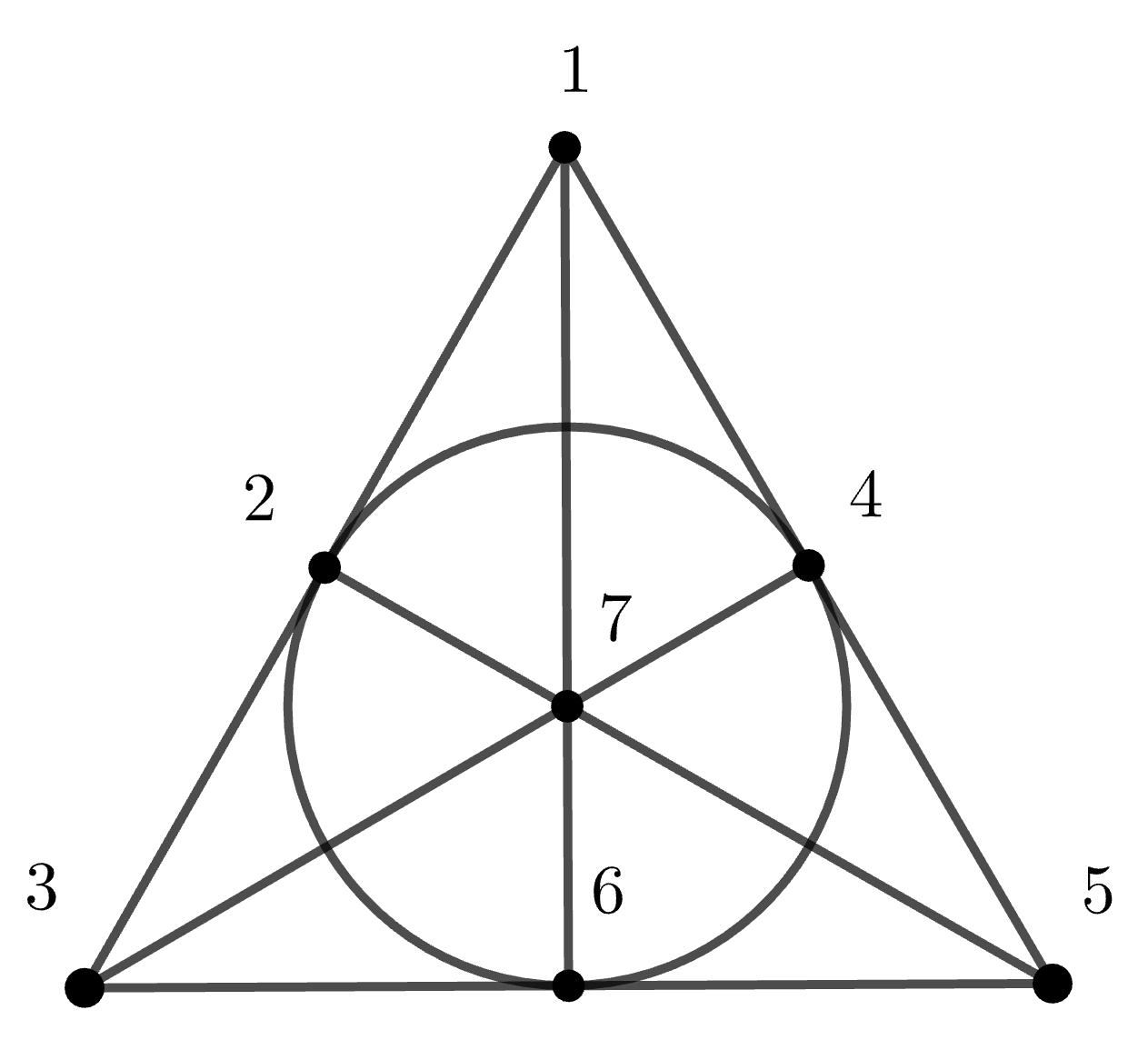}
        \caption{The Fano plane.}
        \label{fig_fano}
    \end{figure}

Incidence varieties are defined by closed conditions, so they contain degenerate configurations as well. For instance, for all $p \in \PP^2$, every incidence variety $\mc{I}(n, L)$ contains $(p,\ldots,p)$. We also observe that, because projective transformations preserve collinearity, the variety $\mc{I}(n, L)$ is invariant under the diagonal action of $\PGL_3$ on $(\PP^2)^n$. The next definitions remove degenerate configurations from $\mc{I}(n, L)$ and remove symmetry.

\begin{defn} \label{def_rs_intro}
    Let $n \in \N$, and let $L$ be a set of collinearity conditions. A \emph{strong realization} of $(n,L)$ over $k$ is an $n$-tuple $(p_1, \ldots, p_n)$ of distinct points in $\PP^2_k$ such that, for each $3$-subset $\ell = \{a,b,c\}$ of $\{1,2,\ldots,n\}$, the points $p_a, p_b, p_c$ are collinear if and only if $\ell$ is a subset of an element of $L$. The set of strong realizations of $L$ defines an open subvariety $\mc{I}^\circ(n,L)$ of $\mc{I}(n,L)$. The \emph{realization space} $\mc{R}(L) = \mc{R}(n,L)$ over $k$ is the moduli space of strong realizations of $L$ up to $\PGL_3$-equivalence.
\end{defn}

Except in a few highly degenerate cases, such as when $n \leq 3$ (Remark \ref{rem_git}), the realization space $\mc{R}(n,L)$ exists as a quasi-projective variety and can be explicitly constructed as follows. By a projective transformation, we prescribe the positions of $4$ points, no $3$ of which are collinear in $L$, and allow the remaining $n-4$ points to vary in $\PP^2$. Then we impose equations and inequations to express $\mc{R}(n,L)$ as a quasi-projective subvariety of $(\PP^2)^{n-4}$, see Definition~\ref{def_rs}. In these non-degenerate cases, the incidence variety $\mc{I}(n,L)$ contains the Zariski closure of $\mc{R}(n,L) \times \PGL_3$, and possibly additional components that parametrize degenerations of $(n,L)$.

All the components of realization spaces $\mc{R}(n,L)$ with $n \leq 9$ are birational to projective spaces \cite{MR1391025}. But Sink recently showed that the realization spaces of some classical $10$-point configurations known as $10_3$-configurations are open subsets of K3 surfaces \cite[Theorem 1.2]{sink}.
In this article, we describe all incidence varieties that can be constructed with $n \leq 10$ points up to birational equivalence. 
In a sense, we show that the K3 surfaces obtained from $10_3$-configurations are the most complicated realization spaces definable with $10$ points.

\begin{thm} \label{thm_main_bir}
    Let $n \leq 10$, and let $\mc{I}(n,L)$ be the incidence variety over $\Q$ determined by a collection $L$ of subsets of $\{1,2,\ldots, n\}$. Then every geometric component of $\mc{I}(n, L)$ is birational over $\bar{\Q}$ to one of the following varieties:
    \begin{enumerate}
        \item a projective space $\PP^N$, where $2 \leq N \leq 2n$,
        \item a genus $1$ curve $\times \PGL_3$,
        \item a K3 surface $\times \PGL_3$.
    \end{enumerate}
\end{thm}

Since the K3 surfaces that appear are described in \cite{sink}, our focus is on types (1) and (2). We provide a gallery of the most interesting examples in Section \ref{sec_census} and provide complete information on the fields of definition of the $\bar{\Q}$-components. Supporting code and data is contained in a repository \cite{site}.

The motivating background for this work is the Mn\"ev-Sturmfels Universality Theorem, which states that every affine variety over $\Q$ is birational to a component of some incidence variety. An early version is due to Mn\"ev \cite{MR802868}, and the version we have stated is due to Sturmfels \cite[Theorem 7]{MR1002208}; see \cite[Section 2.1]{MR1009366} for a proof. This discovery led to a broader search for \emph{universality theorems}. Broadly speaking, these theorems aim to show that all algebro-geometric objects of some kind (i.e., varieties, schemes, singularities) may be constructed within some well-known class of varieties.  A modern version by Lee-Vakil, \emph{Murphy's Law for incidence schemes}, states that for all finite-type $\Z$-scheme singularities $\sigma$, there exists an incidence scheme with a singularity equivalent to $\sigma$ \cite[Theorem 1.3]{MR3114952}. Vakil used this to prove that many other moduli spaces of interest satisfy Murphy's Law \cite{MR2227692}. Further afield, in semialgebraic geometry, there is an analogous universality theorem for configuration spaces of mechanical linkages \cite{MR1923214}.

These universality theorems are constructive, but the constructions used are recursive and rely on a large supply of points. One therefore hopes that when $n$ is small, the incidence varieties that appear are well-behaved. Theorem \ref{thm_main_bir} classifies the incidence varieties that can be constructed using at most $10$ points, and all belong to relatively well-understood classes.

\subsection{Proof outline}
One way to prove these results would be to compute equations for $\mc{I}(n,L)$ for each $1 \leq n \leq 10$ and each of the $2^{2^{n}}$ possible choices of $L$. But there are too many choices for $L$ to work with them all directly, nor do we need to. Using arguments from incidence geometry, we reduce the problem to an explicit study of just $163$ pairs $(n, L)$, called $n_3$-superfigurations.

\begin{defn} \label{def_ls_intro}
    Given $n \in \N$, and a set $L$ consisting of subsets of $\{1,2,\ldots, n\}$, 
    the pair $(n,L)$ is a \textit{linear space} with \emph{point set} $\{1,2,\ldots, n\}$ and \emph{line set} $L$ if the following axioms hold:
\begin{enumerate}
\item any two distinct points belong to exactly one line;
\item each line contains at least two points. 
\end{enumerate}

    An \emph{$n_3$-configuration} is a linear space in which each point is contained in exactly $3$ lines and each line contains exactly $3$ points.
    
    An \emph{$n_3$-superfiguration} is a linear space in which each point is contained in at least $3$ lines, and each line contains at least $3$ points.
\end{defn}

Linear spaces are purely combinatorial objects, defined without reference to an ambient projective plane. The Fano plane is an example, and is easily checked to be a $7_3$-configuration. Linear spaces can also be defined as simple matroids of rank at most $3$. For the matroid point of view, see e.g. \cite{MR2886089}.

Linear spaces account for all possible incidence varieties, in the sense that every component appearing on an incidence variety $(n, L)$ is isomorphic to a component appearing on the incidence variety of a linear space $(n', L')$ for some $n' \leq n$ and some $L'$ (Proposition \ref{prop_ls_reduction}). This justifies the reduction to linear spaces. There are $5749$ linear spaces up to isomorphism with $1 \leq n \leq 10$, a computation first made by D. Betten and D. Glynn; see \cite{MR1670277} and the references therein.

The further reduction to superfigurations follows from a lemma introduced by Glynn in the context of incidence geometry \cite{MR957209}. The idea is that some linear spaces are simple extensions of linear spaces with fewer points. For instance, given a linear space $(n, L)$ consisting of $n$ points with set $L$ of lines, we may define the linear space $(n+1, L)$ consisting of $n + 1$ points and the same lines $L$, placing no conditions on the last point. Then $\mc{I}(n+1, L) \cong \mc{I}(n, L) \times \PP^2$. Similar relationships hold for linear spaces that have a point that belongs to only $1$ or $2$ lines of at least $3$ points; see Lemma \ref{lemma_glynn_ag}. Superfigurations are precisely the linear spaces that do not reduce in this way. Of the $5749$ linear spaces with $n \leq 10$ points, $163$ are $n_3$-superfigurations. 

The most difficult step is then to describe the geometry of the $163$ realization spaces associated to $n_3$-superfigurations. Two cases were previously analyzed in the literature. First, for each of the twelve $n_3$-superfigurations with $n \leq 9$, the realization space has only rational components \cite{MR1391025}. Second, the realization spaces of the ten classical $10_3$-configurations were described by Sink \cite{sink}. Each is birational to $\PP^2$, $\PP^3$, or a K3 surface. The K3 surfaces that appear are quite remarkable: they are defined over $\Q$, are elliptic and modular, have Picard rank $20$, and have discriminant $-7$ (for two different configurations), $-8$, and $-11$.

We classify the remaining $141$ realization spaces in this paper. To handle the large number of cases, we use computer algebra. Each realization space is a Zariski open subset of a subvariety of $\A^{10}$ defined by the simultaneous vanishing of several polynomials. In SageMath \cite{sage}, we compute a Gr\"obner basis over $\mathbb{Q}$ for the defining ideal of each variety viewed as a scheme, then eliminate variables by substitution, producing isomorphic schemes with simpler presentations. This gives a nice model for each realization space over $\Q$. We found that each $\bar{\Q}$-component was rational or a genus $1$ curve.

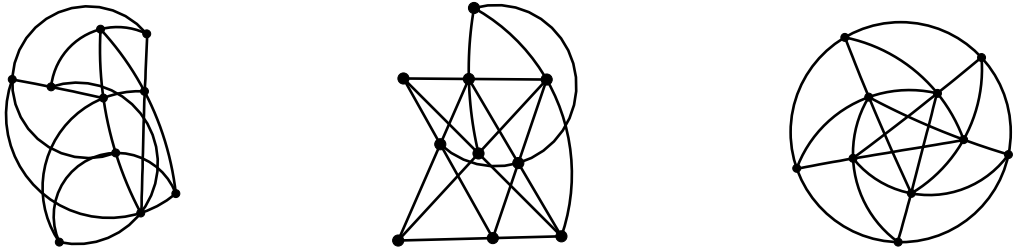
\begin{figure}
    \centering
\begin{tikzpicture}[line cap=round,line join=round,>=triangle 45, scale = 0.5]
\draw [shift={(-2.4320735190440486,4.414795392394823)},line width=1pt]  plot[domain=1.114797733859404:3.0396647886965855,variable=\t]({1*1.7715154520506178*cos(\t r)+0*1.7715154520506178*sin(\t r)},{0*1.7715154520506178*cos(\t r)+1*1.7715154520506178*sin(\t r)});
\draw [shift={(-3.5048955044501597,2.560016739501679)},line width=1pt]  plot[domain=-1.7918011640735898:1.897471955465433,variable=\t]({1*2.148666563490292*cos(\t r)+0*2.148666563490292*sin(\t r)},{0*2.148666563490292*cos(\t r)+1*2.148666563490292*sin(\t r)});
\draw [shift={(6.4004672563962846,5.594924837924639)},line width=1pt]  plot[domain=3.0841248928551255:3.6293941297804393,variable=\t]({1*9.295518328768676*cos(\t r)+0*9.295518328768676*sin(\t r)},{0*9.295518328768676*cos(\t r)+1*9.295518328768676*sin(\t r)});
\draw [shift={(-8.528533475656674,0.897528882283022)},line width=1pt]  plot[domain=0.11156374567573846:0.7470412205566588,variable=\t]({1*7.699071143059578*cos(\t r)+0*7.699071143059578*sin(\t r)},{0*7.699071143059578*cos(\t r)+1*7.699071143059578*sin(\t r)});
\draw [shift={(-3.2090993070868143,4.719285856970105)},line width=1pt]  plot[domain=0.6903284518401642:5.083568208978164,variable=\t]({1*2.019525550797525*cos(\t r)+0*2.019525550797525*sin(\t r)},{0*2.019525550797525*cos(\t r)+1*2.019525550797525*sin(\t r)});
\draw [shift={(286.26226369750015,-5.97801149019916)},line width=1pt]  plot[domain=3.0999955455414554:3.116547769881637,variable=\t]({1*288.16350578675707*cos(\t r)+0*288.16350578675707*sin(\t r)},{0*288.16350578675707*cos(\t r)+1*288.16350578675707*sin(\t r)});
\draw [shift={(-2.6477987759985284,3.8531586617927376)},line width=1pt]  plot[domain=2.791954250806317:5.413211707068524,variable=\t]({1*2.7455718123894095*cos(\t r)+0*2.7455718123894095*sin(\t r)},{0*2.7455718123894095*cos(\t r)+1*2.7455718123894095*sin(\t r)});
\draw [shift={(-2.447098128183799,1.1583066705669414)},line width=1pt]  plot[domain=0.3630720014854452:3.568102891503872,variable=\t]({1*1.6792417399861936*cos(\t r)+0*1.6792417399861936*sin(\t r)},{0*1.6792417399861936*cos(\t r)+1*1.6792417399861936*sin(\t r)});
\draw [shift={(-1.852602633513023,1.9103920230743416)},line width=1pt]  plot[domain=1.5158740285967849:3.7397141353080587,variable=\t]({1*2.5693557401020017*cos(\t r)+0*2.5693557401020017*sin(\t r)},{0*2.5693557401020017*cos(\t r)+1*2.5693557401020017*sin(\t r)});
\draw [shift={(-15.778065418131211,-52.85529013199754)},line width=1pt]  plot[domain=1.3475713473413726:1.3897814379494322,variable=\t]({1*58.60650967040487*cos(\t r)+0*58.60650967040487*sin(\t r)},{0*58.60650967040487*cos(\t r)+1*58.60650967040487*sin(\t r)});
\begin{scriptsize}
\draw [fill=black] (-2.879705869324473,6.128823477297905) circle (3pt);
\draw [fill=black] (-1.6519705837164638,6.005300471653418) circle (3pt);
\draw [fill=black] (-5.22725427499977,4.79367655405185) circle (3pt);
\draw [fill=black] (-4.19439454196237,4.5950496823138876) circle (3pt);
\draw [fill=black] (-2.80400643979664,4.297109374706945) circle (3pt);
\draw [fill=black] (-1.7115586452378522,4.475873559271111) circle (3pt);
\draw [fill=black] (-0.8773257839384141,1.7546854164610328) circle (3pt);
\draw [fill=black] (-1.8108720811068328,1.238255549942332) circle (3pt);
\draw [fill=black] (-3.975904983050613,0.46361075016428055) circle (3pt);
\draw [fill=black] (-2.476587889902069,2.837289450486277) circle (3pt);
\end{scriptsize}
\end{tikzpicture}
\hspace{1in}
        \begin{tikzpicture}[line cap=round,line join=round,>=triangle 45,scale=0.7]
\draw [line width=1pt] (-3.48,1.57)-- (-0.76,1.55);
\draw [line width=1pt] (-3.58,-1.51)-- (-0.48,-1.43);
\draw [line width=1pt] (-3.48,1.57)-- (-1.782745310095254,-1.4636192338089098);
\draw [line width=1pt] (-1.782745310095254,-1.4636192338089098)-- (-0.76,1.55);
\draw [line width=1pt] (-0.76,1.55)-- (-3.58,-1.51);
\draw [line width=1pt] (-3.58,-1.51)-- (-2.238302427420663,1.5608698707898578);
\draw [line width=1pt] (-2.238302427420663,1.5608698707898578)-- (-0.48,-1.43);
\draw [line width=1pt] (-0.48,-1.43)-- (-3.48,1.57);
\draw [shift={(4.669235566240918,1.7357084419589675)},line width=1pt]  plot[domain=2.97081673501844:3.373903505979053,variable=\t]({1*6.9097503326709795*cos(\t r)+0*6.9097503326709795*sin(\t r)},{0*6.9097503326709795*cos(\t r)+1*6.9097503326709795*sin(\t r)});
\draw [shift={(-1.7280578753935256,1.4338310689875375)},line width=1pt]  plot[domain=-2.327793770452036:1.8429346782314393,variable=\t]({1*1.5325700724312*cos(\t r)+0*1.5325700724312*sin(\t r)},{0*1.5325700724312*cos(\t r)+1*1.5325700724312*sin(\t r)});
\draw [shift={(-3.891483545932041,-0.24738771572515816)},line width=1pt]  plot[domain=-0.33369292744940005:1.064332317505902,variable=\t]({1*3.6106497751904447*cos(\t r)+0*3.6106497751904447*sin(\t r)},{0*3.6106497751904447*cos(\t r)+1*3.6106497751904447*sin(\t r)});
\begin{scriptsize}
\draw [fill=black] (-3.48,1.57) circle (3pt);
\draw [fill=black] (-0.76,1.55) circle (3pt);
\draw [fill=black] (-3.58,-1.51) circle (3pt);
\draw [fill=black] (-0.48,-1.43) circle (3pt);
\draw [fill=black] (-1.782745310095254,-1.4636192338089098) circle (3pt);
\draw [fill=black] (-2.238302427420663,1.5608698707898578) circle (3pt);
\draw [fill=black] (-2.780538084472404,0.31980394346094) circle (3pt);
\draw [fill=black] (-2.054897959183674,0.14489795918367332) circle (3pt);
\draw [fill=black] (-1.298712947694629,-0.03737046716276198) circle (3pt);
\draw [fill=black] (-2.14,2.91) circle (3pt);
\end{scriptsize}
\end{tikzpicture}
\hspace{1in}
\begin{tikzpicture}[line cap=round,line join=round,>=triangle 45, scale = 0.5]
\draw [line width=1pt] (-3.444375713792118,5.13216223124338) circle (2.9264689780104063cm);
\draw [shift={(-1.8075708949439155,4.560413465604551)},line width=1pt]  plot[domain=2.5973981670249504:4.085840277487319,variable=\t]({1*2.905911702693075*cos(\t r)+0*2.905911702693075*sin(\t r)},{0*2.905911702693075*cos(\t r)+1*2.905911702693075*sin(\t r)});
\draw [shift={(-244.25166518275202,67.89549490584633)},line width=1pt]  plot[domain=6.0168071715805915:6.033211074520773,variable=\t]({1*249.54137058794134*cos(\t r)+0*249.54137058794134*sin(\t r)},{0*249.54137058794134*cos(\t r)+1*249.54137058794134*sin(\t r)});
\draw [shift={(41.792653648963274,-261.9147123283774)},line width=1pt]  plot[domain=1.7326223453605978:1.7492737168503207,variable=\t]({1*270.3800184326036*cos(\t r)+0*270.3800184326036*sin(\t r)},{0*270.3800184326036*cos(\t r)+1*270.3800184326036*sin(\t r)});
\draw [shift={(-3.213964753601094,3.0590483491736227)},line width=1pt]  plot[domain=1.3345245791466862:2.7862101164005866,variable=\t]({1*3.193888013382345*cos(\t r)+0*3.193888013382345*sin(\t r)},{0*3.193888013382345*cos(\t r)+1*3.193888013382345*sin(\t r)});
\draw [shift={(-5.766138040400408,3.4894457507933554)},line width=1pt]  plot[domain=0.34667921650570505:1.3724929456744344,variable=\t]({1*4.247703253883034*cos(\t r)+0*4.247703253883034*sin(\t r)},{0*4.247703253883034*cos(\t r)+1*4.247703253883034*sin(\t r)});
\draw [shift={(54.30342615010952,30.425370869215268)},line width=1pt]  plot[domain=3.5086141766553016:3.5797304338782165,variable=\t]({1*63.45909885629036*cos(\t r)+0*63.45909885629036*sin(\t r)},{0*63.45909885629036*cos(\t r)+1*63.45909885629036*sin(\t r)});
\draw [shift={(-5.133555074889235,6.810783656474241)},line width=1pt]  plot[domain=-1.0332464834935111:0.07975189085234634,variable=\t]({1*3.8513543779116834*cos(\t r)+0*3.8513543779116834*sin(\t r)},{0*3.8513543779116834*cos(\t r)+1*3.8513543779116834*sin(\t r)});
\draw [shift={(-52.95985117826745,69.42607566602908)},line width=1pt]  plot[domain=5.35101332465303:5.4046756257350665,variable=\t]({1*80.94232056366386*cos(\t r)+0*80.94232056366386*sin(\t r)},{0*80.94232056366386*cos(\t r)+1*80.94232056366386*sin(\t r)});
\draw [shift={(5.175828132275947,23.787371163994344)},line width=1pt]  plot[domain=4.221661204660634:4.42190266871337,variable=\t]({1*20.09374292510796*cos(\t r)+0*20.09374292510796*sin(\t r)},{0*20.09374292510796*cos(\t r)+1*20.09374292510796*sin(\t r)});
\draw [shift={(-2.6826967288733012,6.049623219886496)},line width=1pt]  plot[domain=3.813613699010476:5.6592319276664185,variable=\t]({1*2.5916410965332317*cos(\t r)+0*2.5916410965332317*sin(\t r)},{0*2.5916410965332317*cos(\t r)+1*2.5916410965332317*sin(\t r)});
\begin{scriptsize}
\draw [fill=black] (-4.929313967392828,7.6539035070785015) circle (3pt);
\draw [fill=black] (-1.2944422145881322,7.117610953386004) circle (3pt);
\draw [fill=black] (-0.579385476331471,4.5354616207925) circle (3pt);
\draw [fill=black] (-3.5114595205939096,2.2064622394256523) circle (3pt);
\draw [fill=black] (-6.208277736640267,4.17035856762321) circle (3pt);
\draw [fill=black] (-4.2937079778313505,6.064888533174806) circle (3pt);
\draw [fill=black] (-2.466340757842105,6.164201969043787) circle (3pt);
\draw [fill=black] (-1.77114670675924,4.932715364268423) circle (3pt);
\draw [fill=black] (-3.16153480892497,3.5026018877550977) circle (3pt);
\draw [fill=black] (-4.71082440848107,4.4361481849235185) circle (3pt);
\end{scriptsize}
\end{tikzpicture}
    \caption{A sampling of $10_3$-superfigurations.}
    \label{fig_sf}
\end{figure}

\subsection{Road map}
In Section \ref{sec_algorithms}, we reduce the study of incidence varieties to superfigurations. In Section \ref{sec_census}, we present a census of the varieties associated to the superfigurations on 10 points, proving Theorem \ref{thm_main_bir}. In Section \ref{sec_q}, we pose questions for future research. In an appendix, for lack of a reference, we prove a well-known result that reduces the study of general incidence varieties to the study of linear spaces.

\subsection{Acknowledgements}

The authors wish to thank Ronno Das, Noam Elkies, Sudhir Ghorpade, Brendan Hassett, Robert Lemke Oliver, Joe Silverman, and Elias Sink for helpful discussions. In particular, we thank Noam Elkies for sharing his notes on spaces of $10_3$-configurations with us at an earlier stage of this project. 

Some of the authors obtained partial results in this direction as part of an earlier project on point-counting over finite fields \cite{MR3721583}, which we informally published in the form of an online database \cite{prev_website}. 

Parts of this work appeared in the Ph.D. thesis of the first author \cite[Chapter 2]{isham_thesis}.

Max Weinreich was supported by NSF Grant 2202752.  Nathan Kaplan was supported by NSF grant DMS 2154223.

\section{Reduction to superfigurations} \label{sec_algorithms}

In this section, we show that that every irreducible component of an incidence variety is birational to a variety of the form $\PP^N$ or $I \times \PP^N$, where $I$ is a component of an incidence variety of a superfiguration and $N \geq 0$. (We take the convention that $\PP^0$ is a point.) The ideas in the proof are based on work of Glynn \cite{MR957209}; see also \cite{MR1391025, MR3721583}. Our arguments in this section apply to incidence varieties on any number of points. We take up the case of $n \leq 10$ in Section \ref{sec_census}.

We are interested in studying how a given combinatorial arrangement of points can be realized in a projective plane. Thus, it useful to work with \emph{linear spaces}, which describe arrangements of points and lines without reference to any particular ambient projective plane. This separates the combinatorial aspects of the problem from the algebraic aspects. 
\begin{defn}[\cite{MR1670277}] \label{def_linspace}
A \textit{linear space} $S = (P,L)$ is a pair of sets, a nonempty set $P$ of \emph{points} and a set $L$ of \emph{lines}, such that:
\begin{enumerate}
\item each line is a subset of $P$,
\item any two distinct points belong to exactly one line, 
\item each line contains at least two points. 
\end{enumerate}
A \emph{full line} $\ell$ of a linear space is a line such that $\#\ell \geq 3$. Since a linear space can be recovered from just its point set and its set of full lines, we generally abuse notation by omitting $2$-point lines when we specify line sets.

If $p \in \ell \in L$, we say that $p$ \emph{lies on} $\ell$.

Let $[n] = \{1,2,\ldots, n\}$. If $P = [n]$, this definition of linear space extends Definition \ref{def_linspace} and we write $([n], L) = (n, L)$ for short. The set of linear spaces with point set $[n]$ is denoted $\LS_n$.

A linear space $S$ has an associated set of collinear subsets, defined as follows. A subset $I \subseteq P$ with $\#I \leq 1$ is called \emph{trivially collinear}. A subset $I \subset P$ is \emph{collinear} if it is trivially collinear or $I \subseteq \ell$ for some $\ell \in L$. Given two linear spaces $(P,L)$ and $(P,L')$ with the same set of points, we say $L' \geq L$ if every collinear subset of $P$ in $(P,L)$ is also a collinear subset of $P$ in $(P, L')$.  

Given a pair of linear spaces $S = (P,L)$ and $S' = (P', L')$, a \emph{map of linear spaces} $\sigma: S \to S'$ is a map of sets $\sigma: P \to P'$ such that, for each collinear subset $Q \subseteq P$ in $S$, the image $\sigma(Q)$ is collinear in $S'$. A \emph{(strong) realization} of $S$ in $S'$ is an injective map of linear spaces $\sigma: S \to S'$ with the additional property that, for each non-collinear subset $Q \subseteq P$, the set $\sigma(Q)$ is non-collinear in $S'$. 

\end{defn}

\begin{rmk} There is variation in the literature regarding Definition \ref{def_linspace}. For instance, in \cite{MR1253067}, linear spaces are required to have a non-collinear subset. In \cite{MR3721583}, weak realizations are required to be injective. To avoid confusion, we clarify some consequences of Definition \ref{def_linspace}.
\begin{enumerate}
\item Up to isomorphism, the linear spaces with at most $4$ points are as follows:
\begin{align*}
    &S_{1,1} = (1,\varnothing),\\
    &S_{2,1} = (2,\{\{1,2\}\}), \\
    &S_{3,1} = (3,\{\{1,2\},\{1,3\},\{2,3\}\}),\\
    &S_{3,2} = (3,\{\{1,2,3\}\}),\\
    &S_{4,1} = (4,\{\{1,2\},\{1,3\},\{1,4\},\{2,3\},\{2,4\},\{3,4\}\}), \\
    &S_{4,2} = (4,\{\{1,2,3\},\{1,4\},\{2,4\},\{3,4\}\}), \\
    &S_{4,3} = (4,\{\{1,2,3,4\}\}).\\
\end{align*}
The linear space $S_{1,1}$ is an exceptional case, being the only linear space with empty line set.
\item For each $n$, the relation $\geq$ is a partial order on $\LS_n$ that captures the notion of degeneration. For example, we have
$$S_{4,1} \leq S_{4,2} \leq S_{4,3}.$$
We do not consider comparisons between linear spaces with different point sets.
\item Linear spaces, together with their maps, form a category. Strong realizations are the morphisms of a subcategory. Strong realizations provide us with a notion of ``embedding'' of linear spaces. For example, there are strong realizations of $S_{3,2}$ in $S_{4,2}$ and $S_{4,3}$, but not in $S_{4,1}$; indeed, if $S_{3,2}$ has a strong realization in some linear space $S'$, then $S'$ must have a full line.
\end{enumerate}
\end{rmk}

For the rest of the section, we fix a choice of base field $k$. Our main theorem will take $k = \Q$, but the supporting lemmas apply more generally.

\begin{defn}
The projective plane $\PP^2_k$ over $k$ defines a linear space $(P_k,L_k)$. In algebro-geometric terms, $P_k$ is the set of $k$-points of $\PP^2_k$, and $L_k$ is the set of all vanishing sets of nonzero homogeneous linear polynomials over $k$.

Let $(n,L)$ be a linear space. A \emph{strong realization} of $(n,L)$ over $k$ is a strong realization of $(n,L)$ in $\PP^2_k$. A \emph{weak realization} of $(n,L)$ over $k$ is a linear space map $(n,L) \to \PP^2_k$. The \emph{incidence variety} $\mc{I}(n,L)$ over $k$ is the subvariety of $(\PP^2_k)^n$ parameterizing weak realizations of $(n, L)$ over $k$. The \emph{strong incidence variety} $\mc{I}^\circ(n,L)$ over $k$ is the subvariety of $(\PP^2_k)^n$ parameterizing strong realizations of $(n, L)$ over $k$.
\end{defn}

We note that some linear spaces have no strong realizations over any field $k$, e.g., Pappus and Desargues configurations with a deleted line.

In Definition \ref{def_iv_intro}, we defined incidence varieties associated to arbitrary sets of collinearity conditions. In fact, it suffices to consider incidence varieties associated to linear spaces, as we now explain.

Incidence varieties are algebraic varieties, since they are cut out of $(\PP^2)^n$ by multi-homogeneous equations. Placing homogeneous coordinates $[Z_{i,0} : Z_{i,1} : Z_{i,2}]$ on the $i$-th factor, the ideal of $\mc{I}(n,L)$ is generated by the determinants
\[
\det[abc] \colonequals \begin{vmatrix}
Z_{a,0} & Z_{b,0} & Z_{c,0} \\
Z_{a,1} & Z_{b,1} & Z_{c,1} \\
Z_{a,2} & Z_{b,2} & Z_{c,2} \\
\end{vmatrix},
\]
ranging over $3$-subsets $\{a,b,c\}$ of elements of $L$.

The following proposition justifies working solely with linear spaces rather than arbitrary sets of collinearity conditions. The statement is well-known, but for lack of a reference, we give a proof in the Appendix.

\begin{prop} \label{prop_ls_reduction}
    Let $L$ be a set of subsets of $[n]$. Every component of the incidence variety $\mc{I}(n, L)$ is isomorphic to a component of an incidence variety over $k$ of a linear space on at most $n$ points.
\end{prop}

\begin{example} \label{ex_ls_reduction}
    We verify Proposition \ref{prop_ls_reduction} for the simplest interesting example. Let $n=4$ and $L_0=\{\{1,2,3\},\{2,3,4\}\}$, so
    $$\mc{I}(4,L_0) = \{(p_1,p_2,p_3,p_4):\det[123]=\det[234]=0\}.$$
    Observe that $(4,L_0)$ is not itself a linear space.
    
    The $4$-tuples in $\mc{I}(4,L_0)$ come in two families. The first family consists of sets of $4$ collinear points, while the second family consists of configurations in which $p_2=p_3$. (In the latter, there are no conditions on $p_1$ or $p_4$.) Both families are irreducible $6$-dimensional varieties, so there are $2$ irreducible components. Each appears on the incidence variety of a natural linear space associated to $L$, as follows.
    
    Let $L_1 =\{\{1,2,3,4\}\}$. Then $(4,L_1)$ is a linear space and
    $$ \mc{I}(4,L_1) = \{(p_1,p_2,p_3,p_4):\det[123]=\det[124]=\det[134]=\det[234]=0\}.$$
    
    Let $\Delta_{23}=\{(p_1,p_2,p_3,p_4):p_2=p_3\}$.
    Examining the equations, we see that $\mc{I}(4,L_1) \subseteq \mc{I}(4,L_0)$. The reverse containment is false, since $\Delta_{23} \subseteq \mc{I}(4,L_0)$ and $\Delta_{23} \not\subseteq \mc{I}(4,L_1)$. Nevertheless, we have
    $$\mc{I}(4,L_0) \smallsetminus \Delta_{23} = \mc{I}(4,L_1) \smallsetminus \Delta_{23}.$$
    Indeed, for each configuration in $\mc{I}(4,L_0) \smallsetminus \Delta_{23}$, since $p_2\neq p_3$, the line $\ovl{p_2 p_3}$ is defined, and since $\det[123] = \det[234]=0$, we have $p_1,p_4 \in \ovl{p_2 p_3}$. Taking Zariski closures in $(\PP^2)^4$,
    \begin{align*}
        \{ & \textrm{components of $\mc{I}(4,L_0)$ not contained in $\Delta_{23}$}\} \\ & = \{ \textrm{components of $\mc{I}(4,L_1)$ not contained in $\Delta_{23}$}\}.
    \end{align*}

    Then it remains to classify the irreducible components of $\mc{I}(4,L_0)$ contained in $\Delta_{23}$. But $\Delta_{23}$ is itself an irreducible subset of $\mc{I}(4,L_0)$ that is contained in no other component. We note that $\Delta_{23}$ is abstractly isomorphic to the linear space $(3,\{\{1,2\},\{1,3\},\{2,3\}\})$.
\end{example}

We now turn to superfigurations, the special linear spaces appearing in the birational classification of incidence variety components. The following definition agrees with Definition \ref{def_ls_intro}.

\begin{defn}
Let $n \geq 1$ and $k \geq 3$. An \emph{$n_k$-configuration} is a linear space on $n$ points such that:
\begin{enumerate}
    \item each full line contains exactly $k$ points, and
    \item each point is contained in exactly $k$ full lines.
\end{enumerate}  
An \emph{$n_k$-superfiguration} is a linear space on $n$ points such that:
\begin{enumerate}
    \item each full line contains \emph{at least} $k$ points, and
    \item each point is contained in \emph{at least} $k$ full lines.
\end{enumerate} 
(We do not require that the number of full lines is $n$.) 
\end{defn}

Note that these notions are purely combinatorial, without reference to any ambient space. We only consider $n_3$-configurations and $n_3$-superfigurations in this article, and we henceforth refer to these as simply as configurations and superfigurations. Configurations are a special case of superfigurations. There are no $n_k$-configurations with $k \ge 4$ on fewer than $13$ points.

Our next step is to reduce the study of incidence varieties to those that arise from superfigurations, following the ideas of Glynn \cite{MR957209}. Given distinct $i, j$ with $1 \leq i \leq n$ and $1 \leq j \leq n$, define
$$\Delta_{ij} = \{(p_1, \dots, p_n) : p_i = p_j\} \subset (\PP^2)^n.$$
The \emph{big diagonal} is the variety
$$\Delta = \bigcup_{i \neq j} \Delta_{ij}.$$

The following lemma is a birational version of Glynn's lemma \cite[Lemma 3.14]{MR957209}.

\begin{lemma} \label{lemma_glynn_ag}
Suppose that $n > 1$ and the linear space $S = (n, L)$ is not a superfiguration. Then:
\begin{enumerate}
    \item There exists a linear space $S'= (n-1, L')$ such that, for every component $\mc{C}_0$ of the strong incidence variety $\mc{I}^\circ(S)$, there exists a component
$\mc{C}'_0$ of $\mc{I}^\circ (S')$ such that $\mc{C}_0$ is birational to $\mc{C}'_0 \times \PP^0$, $\mc{C}'_0 \times \PP^1$, or $\mc{C}'_0 \times \PP^2$. 
    \item For every component $\mc{C}$ of the incidence variety $\mc{I}(S)$, there exist $m < n$, a linear space $S' = (m, L')$ and a component $\mc{C}'$ of $\mc{I}(S')$ such that $\mc{C}$ is birational to $\mc{C}' \times \PP^0$, $\mc{C}' \times \PP^1$, or $\mc{C}' \times \PP^2$. 
\end{enumerate}
\end{lemma}

\begin{proof}
First, we prove Claim (1). Since $S$ is not a superfiguration, without loss of generality, we may assume that point $n$ is contained in at most $2$ full lines. Let $S' = (n-1,L')$, where
$$L' = \{ \ell \cap [n-1] : \ell \in L, \; \#(\ell \cap[n-1]) \geq 2 \}.$$
It is easy to check that $S'$ is a linear space.

Let $N$ be the number of lines $\ell' \in L'$ such that $\ell' \cup \{n\}$ is a collinear subset of $S$. Since the number of full lines of $S$ containing $n$ is at most $2$ by hypothesis, we have $0 \leq N \leq 2$. Let $\pi: (\PP^2)^n \to (\PP^2)^{n-1}$ be the map that forgets the last factor. Given a component $\mc{C}'$ of $\mc{I}^\circ(\mc{S}')$, let 
$$\Phi_{\mc{C}'} = \pi^{-1}(\mc{C}') \cap \mc{I}(S).$$
The $\pi$-fiber over a strong realization of $S'$ is parameterized by the choice of image of $n$, which is constrained by $N$ independent linear conditions on $\PP^2$, so $\pi$ gives $\Phi_{\mc{C}'}$ the structure of a $\PP^{2-N}$-bundle over $\mc{C}'$. So $\Phi_{\mc{C}'}$ is birational to $\mc{C}' \times \PP^{2-N}$. 
Now let
$$\Phi^\circ_{\mc{C}'} = \pi^{-1}(\mc{C}') \cap \mc{I}^\circ(S).$$
Every strong realization of $S$ is mapped by $\pi$ to a strong realization of $S'$, so
$$\mc{I}^\circ(S) = \bigcup_{\mc{C}' \subseteq \mc{I}^\circ(S')} \Phi^\circ_{\mc{C}'}.$$
So if $\mc{C}_0$ is a component of $\mc{I}^\circ(S)$, there exists a unique component $\mc{C}'$ of $\mc{I}^\circ(S')$ such that $\Phi^\circ_{\mc{C}'}$ is Zariski dense in $\mc{C}_0$. This proves (1).

We now show that (1) implies (2). Suppose we are given $S = (n,L)$ and a component $\mc{C}$ of $\mc{I}(S)$. Without loss of generality, we may assume that $\mc{C}$ is not contained in the big diagonal $\Delta$, since if it were, then $\mc{C}$ would be isomorphic to a component of an incidence variety of a linear space on strictly fewer points. Similarly, we may assume without loss of generality that generically, elements of $\mc{C}$ are strong realizations of $S$, since otherwise we could replace $S$ by a more degenerate linear space on $n$ points. It follows that $\mc{C} \cap \mc{I}^\circ(S)$ is a component $\mc{C}_0$ of $\mc{I}^\circ(S)$, to which we may apply (1). Then if $S', \mc{C}'_0$ are the linear space and component in the result of (1), we may take $S'$ and the component $\mc{C}'$ of $\mc{I}(S')$ containing $\mc{C}'_0$ in (2).
\end{proof}

\begin{example}
    Let $S$ be the linear space with point set $[5]$ and precisely two full lines $\{ 1,2,5\}, \{3,4,5\}$. Then strong realizations of $S$ are in one-to-one correspondence with strong realizations of the linear space $S' = S_{4,1}$ with point set $[4]$ and no full lines. Indeed any strong realization of $S$ projects to a strong realization of $S'$ by forgetting point $5$, and any strong realization of $S'$ extends to a strong realization of $S$ by setting $p_5 = \ovl{p_1 p_2} \cap \ovl{p_3 p_4}$. Note that weak realizations of $S$ cannot be parametrized in this way, since a weak realization of $S'$ may have more extensions if some of the points $p_1,p_2,p_3,p_4$ coincide or if at least three of these points are collinear. In the language of the proof of Lemma \ref{lemma_glynn_ag}, we have $N = 2$ and each component of $\mc{I}^\circ(S)$ is birational to a $\PP^0$-bundle over a component of $\mc{I}^\circ(S')$. With a little more work, one can show that $\mc{I}^\circ(S)\cong \mc{I}^\circ(S')\cong \PGL_3$.
    
In the same example, if we forget point $1$ instead of point $5$, we see that each component of $\mc{I}^\circ(S)$ is birational to a $\PP^1$-bundle over a component of $\mc{I}^\circ(S'')$ where $S'' \cong S_{4,2}$.

\end{example}

\subsection{Frames and projective transformations} \label{sect_frames}

We now explain the role of the group $\PGL_3$ of projective transformations in simplifying the study of incidence varieties. Applying a projective transformation to a strong realization or weak realization produces another strong realization or weak realization. In Definition \ref{def_rs_intro}, we defined realization spaces $\mc{R}(n,L)$ associated to sets of collinearity conditions as moduli spaces of strong realizations up to $\PGL_3$-equivalence. The following recipe expresses $\mc{R}(n,L)$ as a quasi-projective variety. The construction is not canonical, as it depends on a nice choice of $4$ points in $[n]$. However, different choices produce isomorphic varieties, so there is a well-defined realization space up to isomorphism.

\begin{defn} \label{def_frame}
    Let $S$ be a linear space. A \emph{combinatorial frame} is an ordered choice of $4$ distinct points of $S$ such that no line of $S$ contains at least three of the chosen points. A \emph{framed linear space} is a linear space with a choice of combinatorial frame.

Viewing $\PP^2_k$ as a linear space, the \emph{standard frame} is the combinatorial frame $(x_1, x_2, x_3, x_4)$, where
\begin{align*}
    x_1 &= [0:0:1],\\
    x_2 &= [0:1:0],\\
    x_3 &= [1:1:1],\\
    x_4 &= [1:0:0].
\end{align*}
\end{defn}

\begin{defn} \label{def_rs}
Let $S = (P, L)$ be a linear space on $n \geq 4$ points that admits a combinatorial frame $\wp = (p_1,p_2,p_3,p_4)$. A \emph{framed map} $\rho: S \to \PP^2_k$ relative to $\wp$ is a map of linear spaces that takes $\wp$ to the standard frame, i.e., for all $i = 1,2,3,4$, $\rho(p_i) = x_i$.
The set of framed maps $S \to \PP^2_k$ relative to $\wp$ is naturally parametrized by a subvariety $\mc{F}_\wp(S)$ of $(\PP^2_k)^{n-4}$. The \emph{framed realization space} $\mc{R}_\wp(S)$ with respect to the combinatorial frame $\wp$ is the quasi-projective subvariety of $\mc{F}_\wp(S)$ parametrizing just those framed maps that are strong realizations of $S$ in $\PP^2_k$.

Given combinatorial frames $\wp, \wp'$ of $S$, there is a natural isomorphism $\mc{R}_\wp(S) \cong \mc{R}_{\wp'}(S)$ defined by $\rho \mapsto \mu_\rho \circ \rho$, where $\mu_\rho$ is the unique element of $\PGL_3$ taking $\rho(\wp')$ to the standard frame. The \emph{realization space} $\mc{R}(S)$ is the variety $\mc{R}_\wp(S)$, defined up to isomorphism.
\end{defn}

\begin{rmk} \label{rem_git}
Definition \ref{def_rs} is consistent with the definition of geometric quotient in geometric invariant theory, which is a general (but technical) method for constructing moduli spaces within the category of varieties. To be precise, given a linear space $S$, define $\mc{R}(S) = \mc{I}^\circ(S) \GITQ \PGL_3$, where $\GITQ$ denotes the categorical quotient. If the linear space $S$ contains a combinatorial frame, then basic arguments from geometric invariant theory show that $\mc{R}(S)$ is a geometric quotient, is a fine moduli space, and is isomorphic to the framed realization space from Definition \ref{def_rs}. This gives a canonical variety representing $\mc{R}(S)$. For details, see \cite{sink}.
\end{rmk}

Not all linear spaces admit combinatorial frames, but all superfigurations do. In fact, superfigurations admit a particularly nice type of combinatorial frame for the purposes of studying $\mc{R}(S)$, which we now define.

\begin{defn}
    A \emph{V-shaped frame} is an ordered choice of $5$ distinct points $(p_1,p_2,p_3,p_4,p_5)$ such that $p_1, p_2, p_3$ are collinear, $p_1, p_4, p_5$ are collinear, and $(p_2, p_3, p_4, p_5)$ is a combinatorial frame.
\end{defn}

\begin{lemma} \label{lemma_frame}
Every superfiguration admits a V-shaped frame. 
\end{lemma}

\begin{proof}
Let $S$ be a superfiguration. Choose any point $p_1$ of $S$. By the definition of a superfiguration, there are at least $3$ full lines passing through $p$. Choose any $2$ of these lines and call them $\ell_1$ and $\ell_2$. Choose two other points $p_2, p_3$ on $\ell_1$ and two other points $p_4, p_5$ on $\ell_2$. If there were a line $\ell$ containing $3$ of the points $p_2,p_3,p_4,p_5$, then the linear space axioms imply that $\ell$ would be equal to at least one of $\ell_1$ and $\ell_2$. This would imply that there is a point not equal to $p_1$ in $\ell_1 \cap \ell_2$, which is impossible. So $(p_1, p_2, p_3, p_4, p_5)$ is a V-shaped frame.
\end{proof}

Finally, we put together the various pieces involved in our reduction to studying superfigurations.

\begin{prop} \label{prop_red_to_vee}
    For all $n \in \N$, if $\mc{C}$ is a component of an incidence variety over $k$ on $n$ points, then $\mc{C}$ is birational to one of the following:
    \begin{enumerate}
        \item 
 a projective space $\PP^N$, where $2 \leq N \leq 2n$;
        \item $\mc{C}' \times \PGL_3 \times \PP^N$, where $\mc{C'}$ is a component of $\mc{R}(S')$ for some superfiguration $S'$ on $m\leq n$ points, and $0 \leq N \leq 2(n - m)$.
    \end{enumerate}
\end{prop}

\begin{proof}
We do induction on $n$. By Proposition \ref{prop_ls_reduction}, it suffices to prove this classification for linear spaces $S = ([n], L)$. If $n = 1$, there is only one linear space $S_{1,1}$ up to isomorphism, and $\mc{I}(S_{1,1}) = \PP^2$, which is type (1). Now assume the claim holds for all $1 \leq n' < n$; we shall prove it for $n$. 

If $S$ is not a superfiguration, then again since $n > 1$, Lemma \ref{lemma_glynn_ag} (2) shows that any component of $\mc{I}(S)$ is birational to $\mc{C}' \times \PP^N,\; N \in \{0,1,2\}$, where $\mc{C}'$ is a component of a linear space on strictly fewer points. By the inductive hypothesis, $\mc{C}'$ is of type (1) or (2); thus so is $\mc{C}$.

Now we do a further (descending) induction on the poset $(\LS_n, \leq)$. There is a maximal linear space $S = ([n], \{[n]\})$. It is easy to check directly that every component of $\mc{I}(S)$ is of type (1).

Now, fix $S$ and suppose that the claim holds for all $T > S$ in $\LS_n$. We have already finished the case where $S$ is not a superfiguration, so assume $S$ is a superfiguration. In particular, $S$ has a combinatorial frame $\wp = (p_1, p_2, p_3, p_4)$ by Lemma \ref{lemma_frame}. If the images of these $4$ points are generically in general linear position on $\mc{C}$, then since $\PGL_3$ acts sharply $4$-transitively on frames in $\PP^2$, there is a component $\bar{\mc{C}}$ of $\mc{R}(L)$ such that $\mc{C}$ is birational to $\bar{\mc{C}} \times \PGL_3$, which is type (2). Otherwise, there is a linear condition not prescribed by $L$ that holds everywhere on $\mc{C}$, so $\mc{C}$ is contained in $\mc{I}(L')$, for some $L' > L$. Since $\mc{I}(L') \subseteq \mc{I}(L)$, we see that $\mc{C}$ is a component of $\mc{I}(L')$, which has already been described by the inductive hypothesis.
\end{proof}

We could in principle describe $\mc{R}(S)$ for every superfiguration $S$, but there is a mild computational advantage in working with closed affine schemes as opposed to quasi-projective varieties. For that reason, we introduce a scheme that makes special use of V-shaped frames.

\begin{defn} \label{def_v_scheme}
Consider a superfiguration $S=(n,L)$ such that:
\begin{itemize}
\item the points $(1, 2, 3, 4, 5)$ are a V-shaped frame;
\item the points other than $1, 2, 3$ on the full line containing $1, 2, 3$ are ordered last.
\end{itemize}
We call $S$ a \emph{framed superfiguration}. 

Let $R$ be a commutative ring, assumed to be $\Q$ unless otherwise specified. We associate an affine $R$-scheme $X_S$ to $S$ as follows. Let $n'$ be the number of points of $S$ other than $4, 5$ that are not on the full line containing $1,2,3$. Let $n'' = n - n' - 5$.  
Consider the space of matrices of the form
\[
\begin{pmatrix}
0 & 0 & 0 & 1 & 1 & 1 & \hdots & 1 & 0 & \hdots & 0 \\
1 & 0 & 1 & 1 & 0 & y_1 & \hdots & y_{n'} & 1 & \hdots & 1 \\
1 & 1 & 0 & 1 & 0 & z_1 & \hdots & z_{n'} & w_{1} & \hdots & w_{n''} \\
\end{pmatrix}.
\]
For any $1 \leq i_1, i_2, i_3 \leq n$, let $\det[ i_1 \; i_2 \; i_3 ]$ denote the determinant of the minor formed by columns $i_1, i_2, i_3$, considered as a polynomial in the $y_j, z_j, w_j$. Let $I_S$ be the ideal generated by $\det [i_1 \; i_2 \; i_3]$ for each combinatorially collinear triple $(i_1, i_2, i_3)$ of distinct points of $S$. Define $X_S$ to be the closed subscheme of $\A^{2n'+n''}_R$ with ideal $I_S$. The reduced scheme structure on $X_S$ is denoted $(X_S)_\red$. 
\end{defn}

Informally, it is helpful to think of $X_S$ as a moduli space of weak realizations of $S$ that are strong with respect to the first five points. 

In our applications in Section \ref{sec_census}, we have $n \leq 10$, and it is always possible to select a frame so that $n' =n-5$, $n'' = 0$. For some larger superfigurations, e.g. $13_4$-configurations, there are no $3$-point lines, hence no frames with $n' = n-5$.

\begin{rmk}
Each component of $\mc{R}(S)$ is isomorphic to a Zariski open subset of $(X_S)_\red$. But $X_S$ may have more components, due to the distinction between strong and weak realizations. Indeed it may be the case that $\mc{R}(S)$ is empty while $X_S$ is not. Further, the number of components of $X_S$ depends on the ordering of the points of $S$, since the first $5$ points play a distinguished role in the definition. 
\end{rmk}

\begin{example}
    The following example, related to the Desargues configuration, demonstrates the role that a choice of V-shaped frame plays in this analysis. Let 
\begin{align*}
    L = \{ & \{ 1,2,3 \}, \{1,4,5\}, \{1,6,7\}, \{2,4,8\}, \{3,5,8\}, \\
& \{2,6,9\}, \{3,7,9\}, \{4,6,10\}, \{5,7,10\}\}.
\end{align*}
Let
\begin{align*}
S_1 &= (10, L ), & \textrm{(the Desargues configuration minus one line),}\\
S_2 &= (10, L \cup \{8,9,10\}) & \textrm{(the Desargues configuration),}\\
S_3 &= (11, L \cup \{\{1,8,9 \}, \{1,10,11\}\}),\\
S_{4} &= (11, L \cup \{1,8,9,10,11 \}).
\end{align*}
The scheme $X_{S_3}$ is not empty, as one may check directly.  But if we re-order the points so that the first five points are $(1,8,9,10,11)$, the resulting variety is empty. To see this, we argue as follows. By Desargues' theorem, every weak realization of $S_1$ is a weak realization of $S_2$. So every weak realization of $S_3$ is a weak realization of 
$S_4$. In particular, there are no weak realizations of $S_3$ in which the points $8,9,10,11$ are in general position.
\end{example}

\section{Census} \label{sec_census}

This section takes up the study of incidence varieties $\mc{I}(n,L)$ with $n \leq 10$ points. We prove Theorem \ref{thm_main_bir} by explicitly describing the realization space of each $n_3$-superfiguration with $n \leq 10$. 

The $n < 10$ case is handled by the following theorem.

\begin{thm}[\cite{MR1391025}] \label{thm_census_9}
    Let $S$ be a superfiguration on at most $9$ points. Each $\bar{\Q}$-component of the realization space $\mc{R}(S)$ is a Zariski open subset of a projective space $\PP^N$, where $0 \leq N \leq 2$. 
\end{thm}

The main step in the argument for $n = 10$ is the following computation, which uses computer algebra to describe the affine schemes $X_S$ associated to $10_3$-superfigurations $S$. We omit $10_3$-configurations, as they were analyzed in \cite{sink}.

\begin{comp} \label{comp_comp}
For each $10_3$-superfiguration $S$ that is not a $10_3$-configuration, there is a choice of V-shaped frame such that each $\bar{\Q}$-component of $(X_S)_\red$ is birational to one of the following:
    \begin{enumerate}
        \item a projective space $\PP^N_{\bar{\Q}}$ $(0 \leq N \leq 2)$, 
        \item a genus $1$ curve.
    \end{enumerate}
\end{comp}

\begin{proof}
We conducted the following calculations in Magma and SageMath \cite{MR1484478, sage}; calculations are stored in the code repository \cite{site}. 

We imported a complete list of $10_3$-superfigurations up to isomorphism from the database \cite{prev_website}; see the associated article \cite[Section 3]{MR3721583}. There are $141$ isomorphism classes that are not classes of $10_3$-configurations. We chose a framed representative $S$ of each isomorphism class. For each $S$:
\begin{enumerate}
    \item[(i)] We computed the affine $\Q$-scheme $X_S$, presented as a polynomial ring $A_S$ together with an ideal $I_S \subset A_S$ given by a list of generators.
    \item[(ii)] We computed a Gr\"obner basis of $I_S$, then performed some elementary substitutions over $\Q$ (if possible) to obtain a polynomial ring $\tilde{A}_S$ in fewer variables and an ideal $\tilde{I}_S \subset \tilde{A}_S$, presented in terms of generators, such that $\tilde{A}_S/\tilde{I}_S \cong A_S/I_S$ as $\Q$-algebras. 
    \item[(iii)] We constructed $\tilde{X}_S = (\Spec \tilde{A}_S/\tilde{I}_S)_{\red}$, i.e., the $\Q$-scheme defined by the simplified ideal. Note that $\tilde{X}_S \cong (X_S)_{\red}$ as $\Q$-schemes.
    \item[(iv)] We calculated the $\Q$-components of $\tilde{X}_S$. For each $\Q$-component $J$, we calculated $\dim J$.
\end{enumerate}
We recorded a list $\mc{J}$ of the $\Q$-components $J$ that appeared in this way, sorted by dimension. This showed that $\dim J \leq 2$ for all $J \in \mc{J}$.

For each $J \in \mc{J}$ with $\dim J = 2$, we checked that $J \cong \A^2_\Q$ as $\Q$-schemes. 

For each $J \in \mc{J}$ with $\dim J = 1$, we computed the $\bar{\Q}$-components of $J$. For each $\bar{\Q}$-component, we computed the geometric genus. 
We found three possibilities: each $J$ was birational to $\PP^1_\Q$, was a genus $1$ curve, or split into two Galois-conjugate $\bar{\Q}$-components isomorphic to $\A^1_{\bar{\Q}}$.

For each $J \in \mc{J}$ with $\dim J = 0$, it is automatic that the geometric components are points.
\end{proof}

Gathering Computation \ref{comp_comp}, the $n \leq 9$ case, and the $10_3$-configuration case, we obtain the following classification.

\begin{thm}\label{thm_census}
    Let $S$ be a superfiguration on at most $10$ points. Each $\bar{\Q}$-component of the realization space $\mc{R}(S)$ is a Zariski open subset of one of the following varieties:
    \begin{enumerate}
        \item a projective space $\PP^N$ $(0 \leq N \leq 3)$, 
        \item a genus $1$ curve,
        \item a K3 surface.
    \end{enumerate}
\end{thm}

\begin{proof}
The $n < 10$ case is Theorem \ref{thm_census_9}. 
If $S$ is a $10_3$-configuration, then $\mc{R}(S)$ is a Zariski open subset of $\PP^2$ or a K3 surface \cite{sink}.
If $S$ is a $10_3$-superfiguration that is not a $10_3$-configuration, then with the choice of combinatorial frame used in Computation \ref{comp_comp}, since each $\bar{\Q}$-component of $\mc{R}(S)$ is a Zariski open subset of a $\bar{\Q}$-component of $(X_S)_{\red}$, each $\bar{\Q}$-component of $\mc{R}(S)$ is a either a Zariski open subset of a projective space $\PP^N$ with $0 \leq N \leq 3$, or a genus $1$ curve.
\end{proof}

We now prove our main theorem.

\begin{proof}[Proof of Theorem \ref{thm_main_bir}]
    The claim is that if $n \leq 10$ and $L$ is a set of subsets of $[n]$, then every geometric component of $\mc{I}(n, L)$ is birational over $\bar{\Q}$ to one of the following varieties:
    \begin{enumerate}
        \item a projective space $\PP^N$, where $2 \leq N \leq 2n$,
        \item a genus $1$ curve $\times \PGL_3$,
        \item a K3 surface $\times \PGL_3$.
    \end{enumerate}
    Since $\PGL_3$ is itself birational to $\PP^8$, this is Proposition \ref{prop_red_to_vee} combined with Theorem \ref{thm_census}.
\end{proof}

 In Section \ref{sec_examples_9} and \ref{sec_examples_10}, we describe the main characteristics of the set of realization space components and catalogue some examples of interest.

\subsection{Superfigurations on up to $9$ points} \label{sec_examples_9}

The realization spaces of the $n_3$-superfigurations for all $n \leq 9$ were studied, albeit in somewhat different language, in \cite{MR1391025, MR3721583}. Some of these spaces are well-known, especially the $n_3$-configurations of Fano, M\"obius-Kantor, and Pappus. For completeness, we record some previously known facts regarding these incidence varieties.

\subsubsection{The Fano plane}
The Fano plane is the unique $7_3$-superfiguration and the unique $7_3$-configuration, up to isomorphism. A framed representative is given by
    $$S_{\Fano} = (7, \{\{1,2,3\}, \{1,4,5\}, \{1, 6, 7\}, \{3,4,7\}, \{ 3,5,6\}, \{2,5,7\}, \{2,4,6\}\} ).$$
    The matrix associated to $S_{\Fano}$ to construct $X_{S_{\Fano}}$ in Definition \ref{def_v_scheme} is
				\[
				\begin{pmatrix}
					0 & 0 & 0 & 1 & 1 & 1&1\\
					1 & 0 & 1 & 1 & 0  &y_1 & y_2 \\
					1 & 1 & 0 & 1 & 0 & z_1 & z_2
				\end{pmatrix}.
				\]
	Each line of $S_{\Fano}$ contributes a determinant to the defining ideal of $X_{S_{\Fano}}$. Two of the determinants are identically $0$ because of the choices of the first five columns. The remaining determinants are $-y_1 + y_2 + z_1 - z_2$, $y_1-1$, $y_2$, $-z_2 + 1$, and $z_1$.

    Over the base ring $R=\Z$, we may substitute $y_1=1,z_2=1,y_2=0,z_1=0$ to obtain the isomorphic scheme $\Spec \Z/\langle 2 \rangle$; indeed, setting all five determinants to $0$ gives a system of equations that reduces to the single equation $-2=0$. So $X_{S_{\Fano}}$ has no points over $\Q$ (i.e., over $\Q$, the determinants generate the unit ideal), but there is a point over $\F_2$.
    
    The expected dimension of $X_{S_{\Fano}}$ over any field is $-1$, since there are $5$ conditions on $4$ variables.
While $X_{S_{\Fano}}$ is indeed empty over $\Q$, we see that there may be dimension jumps in some characteristics. 

\subsubsection{M\"obius-Kantor}
    The M\"obius-Kantor configuration $S_{\MK}$ is the unique $8_3$-superfiguration and the unique $8_3$-configuration, up to isomorphism. 
    A framed representative is given by
    {
    \[
    S_{\MK} = (8, \{\{1,2,3\}, \{1,4,5\}, \{5, 6, 7\}, \{1,7,8\}, \{ 3,5,8\}, \{2,6,8\}, \{3,4,6\},\{2,4,7\} \} ).
    \]
    }
    The ideal defining $X_{S_{\MK}}$ over $\Q$ is generated by $8$ determinants, one for each $3$-point line; again, $2$ are trivial. So $X_{\MK}$ is defined by an ideal with $6$ generators in a polynomial ring of $6$ variables. It is easy to eliminate all but one variable by substitution, yielding
	$$
	X_{S_{\MK}} \cong \Spec \frac{\Q[z]}{(z^2 - z + 1)}.
	$$
The dimension is $0$, as expected. The points are defined over $\Q(\sqrt{-3})$, but not over $\Q$. This  reflects that the M\"obius-Kantor configuration is unrealizable over $\R$.

\subsubsection{The $9_3$-superfigurations}

    There are ten $9_3$-superfigurations up to isomorphism. Their realization spaces were described in \cite{MR1391025} for applications to enumerative geometry over finite fields. We recall some facts about these spaces.
    
    The Pappus configuration $S_{\Pappus}$ is a $9_3$-configuration.
    The ideal defining $X_{S_{\Pappus}}$ over $\Q$ is generated by $9$ determinants, of which $2$ are trivial. Since $X_{S_{\Pappus}}$ is defined by polynomials in $8$ variables, the expected dimension is $1$. But by the Pappus Theorem, any one of the $9$ lines is forced by the simultaneous occurrence of the others, so in fact $\dim X_{S_{\Pappus}} = 2$. It is birational to $\PP^2$. 
    
    The other $\bar{\Q}$-components associated to $9_3$-superfigurations are each birational to $\PP^0$ or $\PP^1$. The field of definition of each $\PP^1$ is $\Q$, and the fields of definition of the $\PP^0$'s are $\Q$, $\Q(\sqrt{-3})$, $\Q(\sqrt{5})$, and $\Q(\sqrt{-1})$. Because all these fields are quadratic, for each of these affine varieties $X_S$, the number of points of $X_S$ over $\F_q$ is given by a quasipolynomial in $q$ \cite{MR1391025}.

\subsection{Superfigurations on $10$ points} \label{sec_examples_10}

\subsubsection{Dimension $3$ components}
There are ten $10_3$-configurations. The most famous is the Desargues configuration, which appears in the theorem of the same name:
\begin{align*}
S_{\textrm{Desargues}} = (10, \{ & \{ 1,2,3 \}, \{1,4,5\}, \{1,6,7\}, \{8,9,10\}, \{2,4,8\}, \\
&\{3,5,8\}, \{2,6,9\}, \{3,7,9\}, \{4,6,10\}, \{5,7,10\}\}).
\end{align*}

The realization space of the Desargues $10_3$-configuration has expected dimension $2$, but is in fact birational to $\PP^3$. This is similar to what we saw with the Pappus configuration. Desargues is the sole superfiguration with at most $10$ points and a $3$-dimensional realization space. It can be drawn with straight lines, reflecting the fact that it has realizations over $\R$.

\begin{figure}[h]
\includegraphics[height=1.7in]{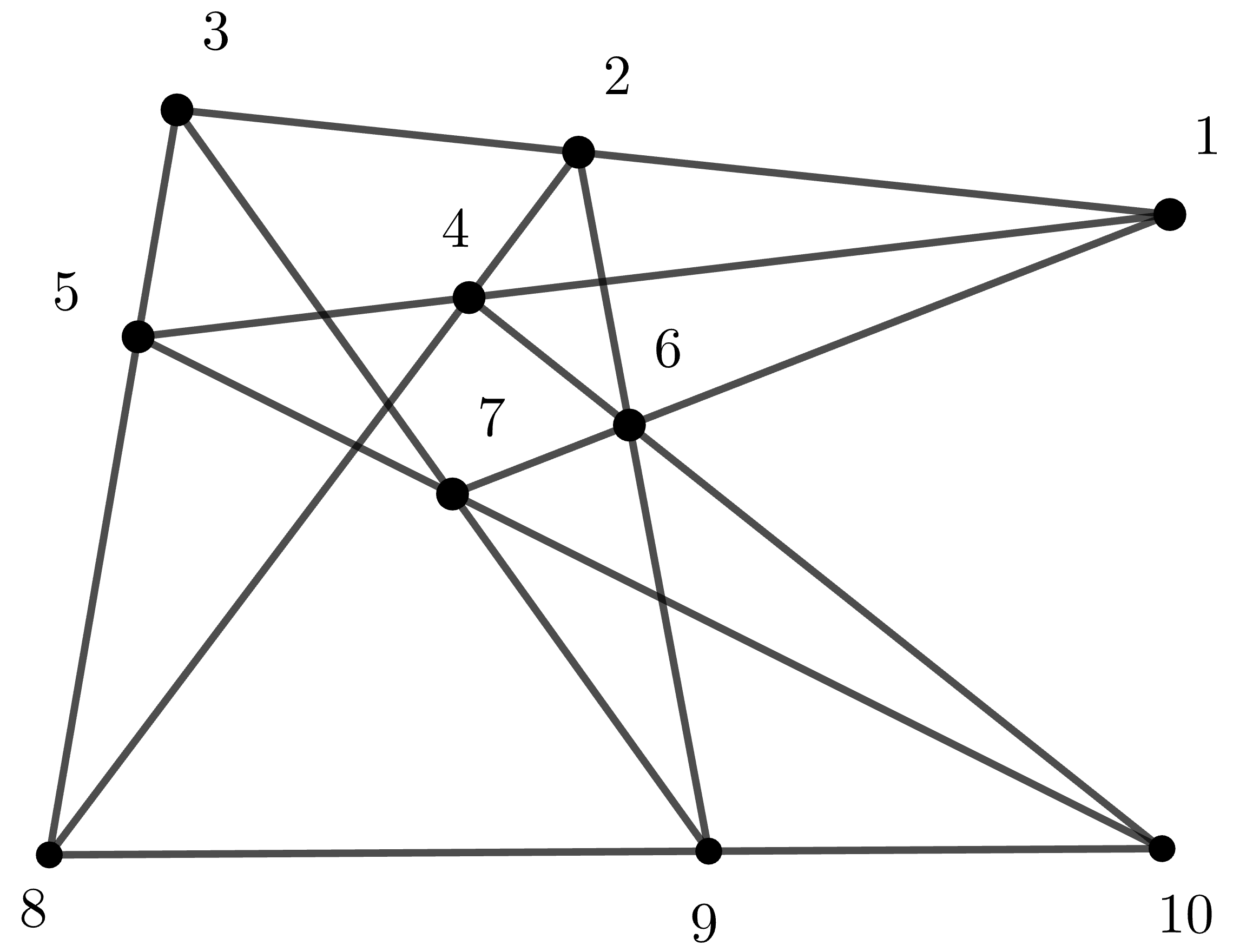}
\caption{The Desargues configuration.}
\label{fig_desargues}
\end{figure}

\subsubsection{Dimension $2$ components}

There are four $10_3$-configurations for which the realization space is birational to a K3 surface. Four $10_3$-configurations have realization varieties that are birational to $\PP^2$ \cite{sink}. We show here that there is precisely one superfiguration that is not a $10_3$-configuration, yet has a $2$-dimensional realization space.

Computation \ref{comp_comp} found that, as $S$ ranged over our list of framed superfigurations that are not $10_3$-configurations, precisely two of the varieties $\tilde{X}_S$ had a $2$-dimensional $\bar{\Q}$-component, and each was birational to $\PP^2$. These components are surprising because the maximum \emph{expected} dimension of these varieties is $1$. 

Let us take a closer look at each. Recall that, while the list of components of $\tilde{X}_S$ depends on the choice of V-shaped frame of $S$, we at least know that $\mc{R}(S)$ may be embedded as a Zariski open (but not necessarily dense) subset of $\tilde{X}_S$. So if $\mc{R}(S)$ is $2$-dimensional, then $\tilde{X}_S$ has a $2$-dimensional component.

\begin{example}
The first $2$-dimensional component of a variety $\tilde{X}_S$ appeared for the \emph{special Desargues superfiguration} (Figure \ref{fig_spdesargues}), defined by
\begin{align*}
S = (10, \{ &\{1,8,9,10\}, \{ 1,2,3 \}, \{1,4,5\}, \{1,6,7\}, \{8,9,10\}, \{2,4,8\}, \\
&\{3,5,8\}, \{2,6,9\}, \{3,7,9\}, \{4,6,10\}, \{5,7,10\}\}).
\end{align*}

\begin{figure}[h]
\includegraphics[height=1.7in]{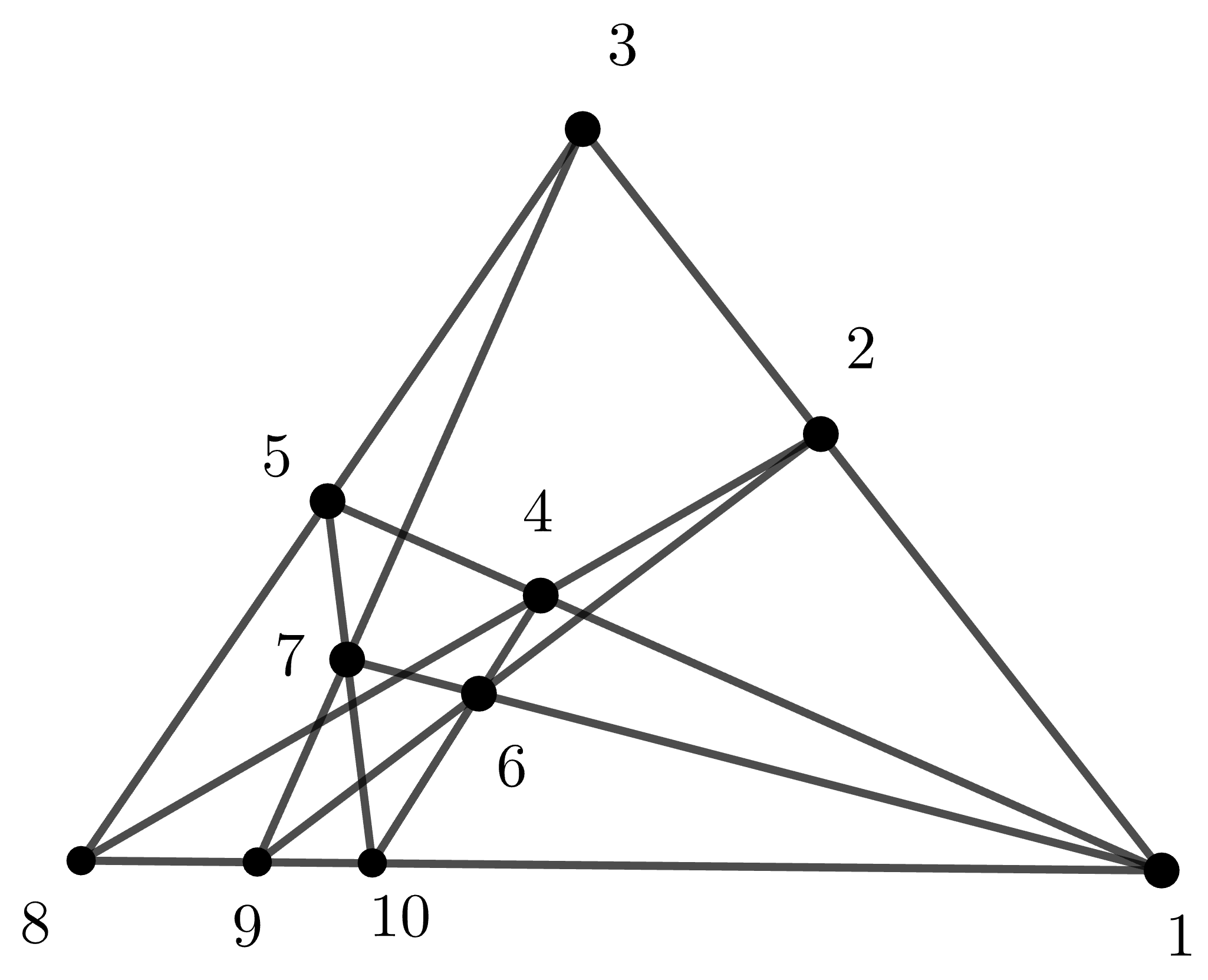}
\caption{The special Desargues superfiguration.}
\label{fig_spdesargues}
\end{figure}

The superfiguration $S$ plays an important role in the moduli problem for Desargues configurations \cite{MR1924759}. One may check directly that $\mc{R}(S)$ is birational to $\PP^2$, even though the expected dimension of $\mc{R}(S)$ is $1$. This is a manifestation of Desargues' theorem. Specifically, observe that $S$ is the degeneration of the Desargues configuration in which point $10$ lies on the line containing $1,2,3$, so the realization space $\mc{R}(S)$ is defined by one more algebraic condition than the realization space of the Desargues configuration, which is $3$-dimensional.
\end{example}

\begin{example}
    The other $2$-dimensional component of a variety $\tilde{X}_S$ appeared for the superfiguration
    \begin{align*}
S = (10, \{ &\{ 1,2,3 \}, \{1,4,5\}, \{1,6,7\}, \{1,8,9\}, \{2,4,10\}, \\
&\{2,6,8\}, \{3,4,8\}, \{3,6,10\}, \{5,6,9\}, \{5,7,8\}, \{7,9,10\}\}).
\end{align*}
\end{example}
Here, even though $\tilde{X}_S$ has a component birational to $\PP^2$, in fact $\mc{R}(S)$ has dimension $1$, consistent with the expected dimension. The $2$-dimensional component of $\tilde{X}_S$ appears because of the degeneration $S'$ of $S$ where points $4,6,8,10$ coincide. Any assignment of the remaining points $7,9$ on the line containing $1,4,5$ (i.e., $y=z$) defines a framed map of $S$ to~$\PP^2$.

\subsubsection{Dimension $1$ components}

The geometric components of dimension $1$ are algebraic curves of genus $0$ or $1$. 

Curves of genus $0$ are rational. The arithmetic of these curves determines realizability of superfigurations over specific fields. Our computation shows that the curves of genus $0$ each have field of definition $\Q$, $\Q(\sqrt{-3})$, or $\Q(\sqrt{5})$. These last two fields occur only for pairs of lines that split over a quadratic extension.

The genus $1$ curves are defined over $\Q$, and we may again consider their arithmetic to decide realizability over specific fields. We find that each genus $1$ curve that appears admits a $\Q$-rational point, so they are elliptic curves over $\Q$ (after choosing a base point). Their Cremona labels are 14a4, 11a3, 15a8, and 37a1. Each of these elliptic curves is a model for a modular curve, namely $X_1(11), X_1(14),$ $X_1(15),$ and $X_0(37)$ modulo the Fricke involution given by $\tau \mapsto -\frac{1}{37\tau}$. A geometric explanation for the appearance of $X_1(11)$ was suggested to us by Noam Elkies.

\begin{example}[the modular curve $X_1(11)$]
    Given an elliptic curve $E$ with base point $O$ and a marked 11-torsion point $P \neq O$, define $P_0 = O, P_1 = P, P_2 = [2]P\ldots , P_{10} = [10]P$. The sets $\{i,j,k\}$ for which $i+j+k \equiv  0 \pmod{11}$ are the lines of an $11_3$-superfiguration with point set $\{0,1,\ldots,11\}$.
    
	Given $0 \le i \le 10$, let $T_i$ be the induced linear space on $\{0,1,2,3,\ldots,10\} \smallsetminus \{i\}$. For all $i \neq 0$, $T_i$ is isomorphic to the $10$-point superfiguration depicted in Figure \ref{fig_s95}, defined by
	\begin{align*}
		S = (10,\; &\big\{\{1, 2, 3\}, \{1, 4, 5\}, \{1, 6, 7\}, \{1, 8, 9\}, \{2, 4, 10\}, \{2, 5, 9\},\\
		& \{2, 6, 8\}, \{3, 5, 6\}, \{3, 9, 10\}, \{4, 7, 8\}, \{5, 7, 10\}\big\}).
	\end{align*}
    The variety $X_S$ corresponding to $S$ is birational to an elliptic curve with Cremona label 11a3 and is a model for the modular curve $X_1(11)$. Further, $T_0$ is the $10_3$-configuration for which $\mc{R}(T_0)$ is birational to a K3 surface of discriminant $-11$ \cite{sink}.
    	\begin{figure}[h]
\begin{center}
\includegraphics[height=2in]{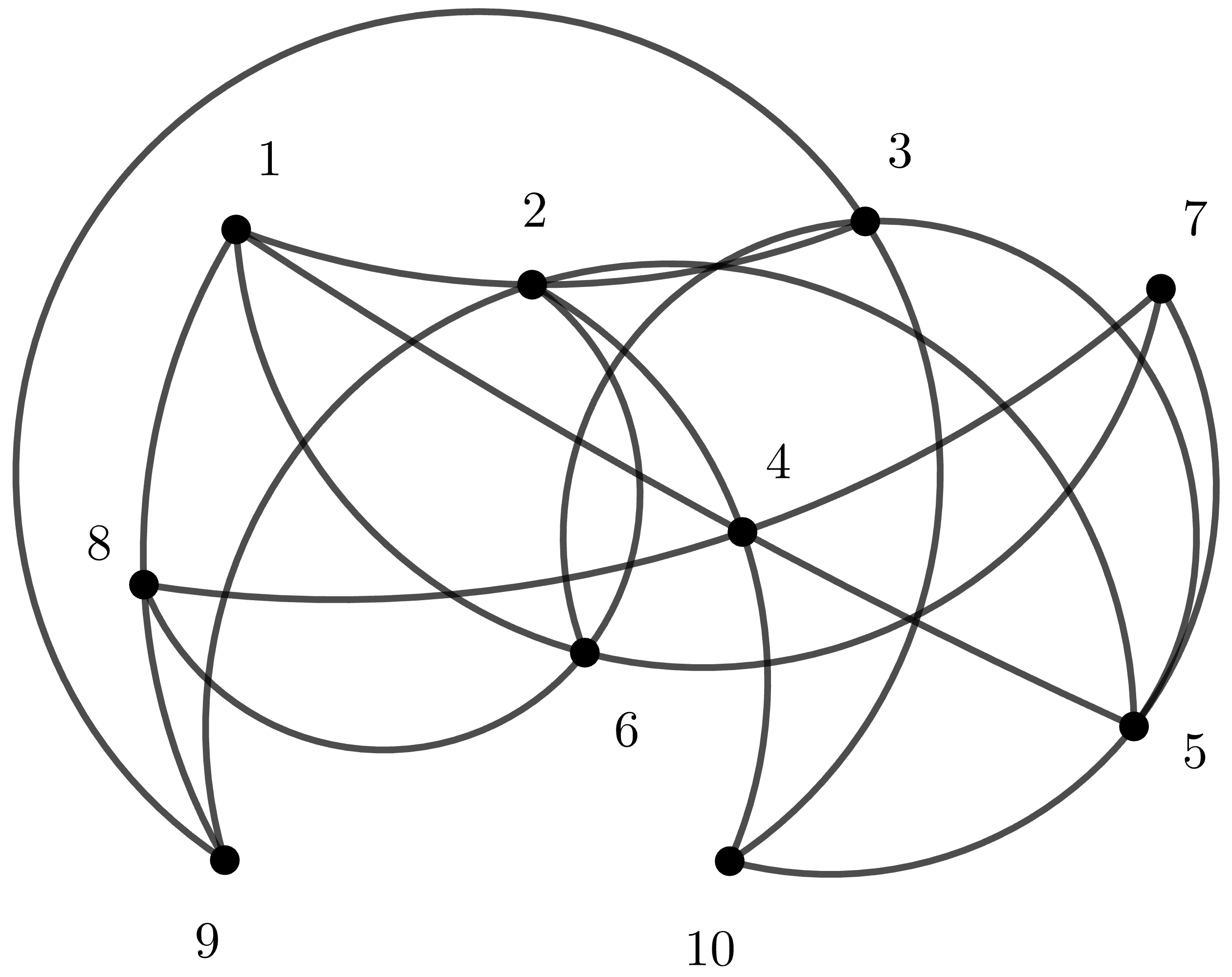}
\end{center}
\caption{A superfiguration $S$ related to a modular curve.}
\label{fig_s95}
\end{figure}
\end{example}

It would be interesting to see geometric reasons for the other modular curves to appear.

\subsubsection{Dimension $0$ components}
The $\bar{\Q}$-components of dimension $0$ are points. The $\Q$-components are important from the perspective of arithmetic and enumerative problems. For all superfigurations with strictly fewer than $10$ points, the $\bar{\Q}$-components of dimension $0$ have field of definition $\Q$, $\Q(\sqrt{-1})$, $\Q(\sqrt{-3})$, or $\Q(\sqrt{5})$.

The possibilities with $n = 10$ are more complicated, but in all cases, the field of definition has degree at most $4$ over $\Q$. We give a minimal polynomial for each: 
\[
\begin{matrix}
& x^2 - 2 & & x^3 - x -1, & & x^4 - x^3 +x^2-x + 1\\ 
& x^2 - 5 & & x^3 - x^2 + x +1 \\
& x^2 + 1 & & x^3 - 5x^2 + 6x - 1 \\
& x^2 + 3 \\
& x^2 + 7 \\
\end{matrix}
\]

\begin{example}\label{s_3_ex}
	Consider
\begin{align*}
		S = (10, \big\{
        & \{1, 2, 3\},
 \{1, 4, 5\},
 \{1, 6, 7\},
 \{1, 8, 9\},
 \{2, 4, 10\},
 \{2, 5, 7\}, \\
 &\{2, 6, 8\},
 \{3, 4, 7\},
 \{3, 5, 8\},
 \{4, 6, 9\},
 \{5, 6, 10\},
 \{7, 9, 10\}
        \big\} ).
	\end{align*}
    Then 
    $$\tilde{X}_S \cong \Spec \frac{\Q[x]}{(x^3 - x^2 + x +1)}.$$
    The number field $\Q[x]/(x^3 - x^2 + x +1)$ has discriminant $-44$ and is not Galois.
\end{example}

\begin{example}
There is exactly one superfiguration for which the realization space has a quartic field of definition:
	\begin{align*}
		S = (10, \big\{
        &\{1, 2, 3\}, \{1, 4, 5\}, \{1, 6, 7\}, \{1, 8, 9\}, \{2, 4, 10\}, \{2, 5, 9\}, \\
        &\{2, 6, 8\}, \{3, 4, 7\}, \{3, 6, 10\}, \{4, 6, 9\}, \{5, 7, 8, 10\}
        \big\}).
	\end{align*}
    For this superfiguration,
    $$\tilde{X}_S \cong \Spec \frac{\Q[x]}{(x^4 - x^3 + x^2 - x + 1)}.$$
Note that $x^4 - x^3 + x^2 - x + 1$ is the 10\textsuperscript{th} cyclotomic polynomial.  The discriminant of the cyclotomic field it generates is $125$, which is the largest discriminant that appears among the number fields arising in our classification.
\end{example}

\begin{example} \label{ex_starfish}
    The \emph{starfish} is the unique superfiguration class with 10 points and a 5-point line; see Figure \ref{fig_starfish}. It is defined by
	\begin{align*}
		S_{\Star} = ( 10, & \; \big\{\{1, 2, 3\}, \{1, 4, 5\}, \{1, 6, 8\}, \{1, 7, 9\}, \{2, 4, 10\},\\ 
  & \{2, 5, 9\}, \{2, 6, 7\},
		\{3, 4, 6\},  \{4, 8, 9\}, \{6, 9, 10\}, &\\
        & \{3, 5, 7, 8, 10\}\big\} ).
	\end{align*}
	The starfish is the incidence structure formed by the ten lines connecting the vertices of a regular pentagon in the real affine plane. Projectivizing, the five pairs of parallels each intersect at the line at infinity, which gives us the required 5-point line. One can further check that any strong realization of $S_{\Star}$ is projectively equivalent over $\C$ to a real regular pentagon. 
\begin{figure}[h]
\begin{center}
\includegraphics[height=1.8in]{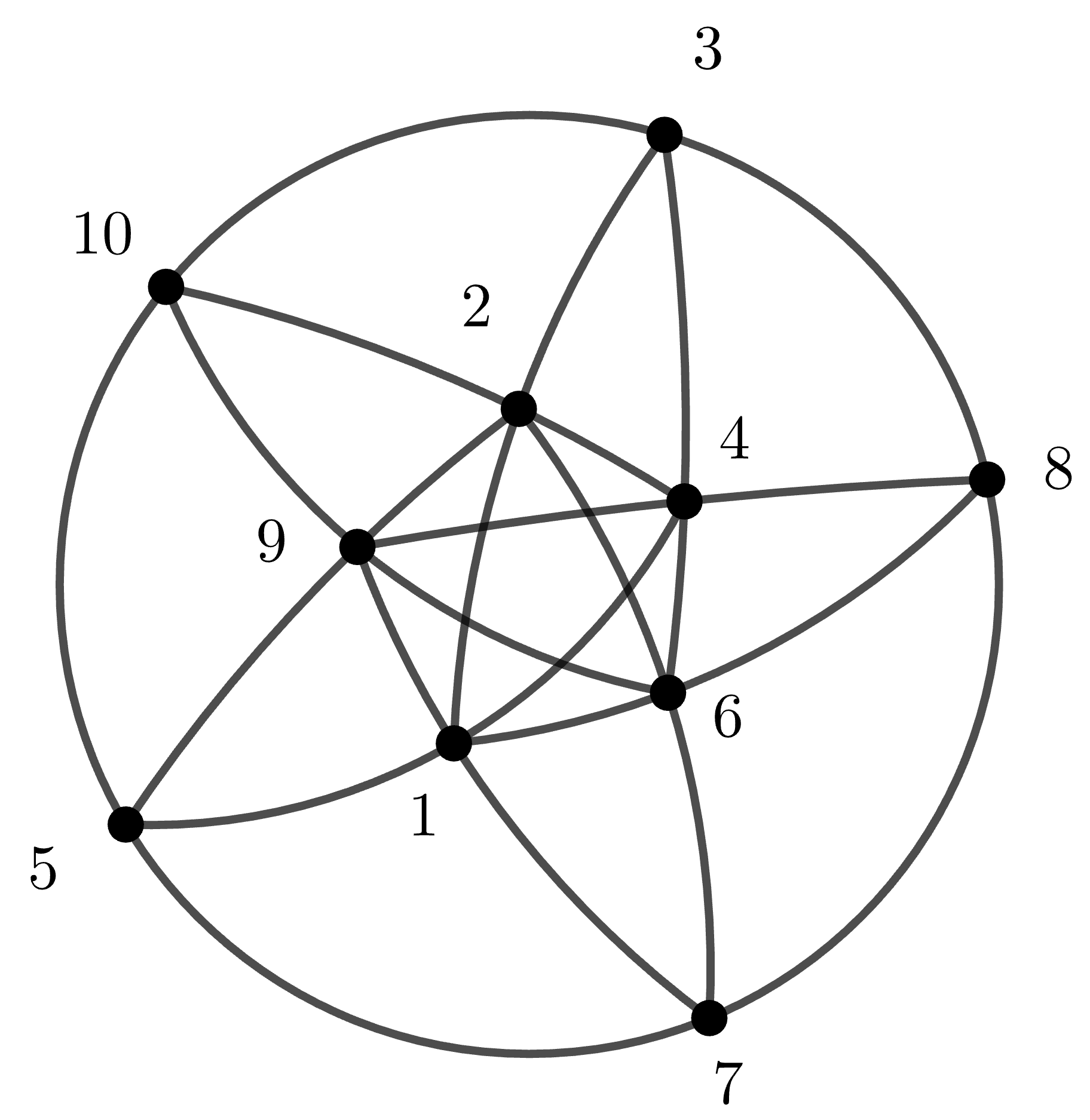}
\end{center}
\caption{The starfish $S_{\Star}$.}
\label{fig_starfish}
\end{figure}

The ideal defining $X_{S_{\Star}}$ is generated by 20 determinants: one for each 3-point line, and $\binom{5}{3} = 10$ for the 5-point line. After taking into account dependencies of the determinants from the $5$-point line, the expected dimension is $-1$. We compute
	$$
	X_{S_{\Star}} \cong \Spec \frac{\Q[x]}{(x^2 + x - 1)}.
	$$ 
The field of definition of $X_{S_{\Star}}$ is $\Q(\sqrt{5})$, reflecting the appearance of $\sqrt{5}$ in the construction of a regular pentagon.
\end{example}

\subsubsection{Unrealizable superfigurations}

There is one ${10}_3$-configuration $S$ which has no strong realizations over \emph{any} field.
To prove its unrealizability, one can show that $S$ is two Pappus configurations pasted together, minus some lines \cite{MR1706508}. Such linear spaces are sometimes called \emph{anti-Pappian} because they do not embed in Pappian projective planes (that is, projective planes over fields). There are many more unrealizable superfigurations beyond the well-known unrealizable $10_3$-configuration.   

\begin{example} \label{ex_anti_pappian}
     Consider the superfiguration
    \begin{align*}
        S = \big( 10, \{ &\{1, 2, 3\}, \{1, 4, 5\}, \{1, 6, 7\}, \{2, 6, 10\}, \{4, 9, 10\}, \{3, 6, 8\}, \\
        & \{7, 8, 9\}, \{2, 4, 8\}, \{2, 5, 7\}, \{5, 6, 9\}, \{3, 7, 10\}  \} \big).
    \end{align*}
    
    \begin{figure}[h]
        \centering
        \includegraphics[height=2in]{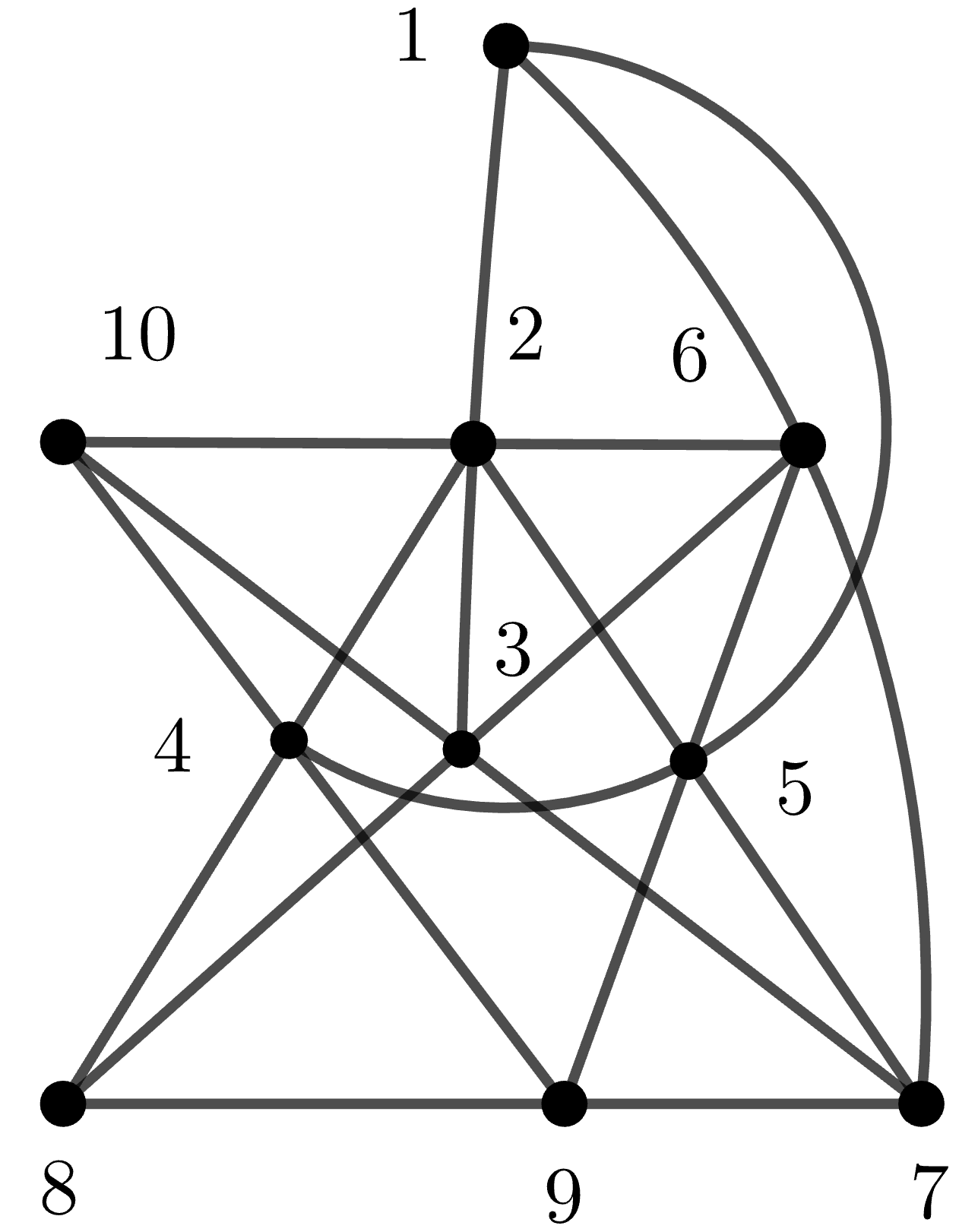}
        \caption{A superfiguration forbidden by Pappus' Theorem.}
        \label{fig_unrealizable}
    \end{figure}
    
    The points $2, \hdots, 9$ form a Pappus configuration minus the line $\{3, 4, 5\}$, and the lines $\{1, 2, 3 \}, \{1, 4, 5\}, \{1,6,7\}$ were chosen to ensure that each point lies on at least 3 lines; see Figure \ref{fig_unrealizable}. Any weak realization of $S$ satisfies Pappus' Theorem, so $\det[345]$ vanishes. It follows that $S$ has no strong realizations over any field.
\end{example}

\section{Questions for future research} \label{sec_q}

Since a linear space is a combinatorial blueprint for a variety, it is natural to ask about connections between combinatorial properties of linear spaces and the geometry of the corresponding incidence varieties. For instance, automorphisms of linear spaces induce automorphisms of their incidence varieties. But given the diversity of incidence varieties, it seems difficult to formulate any general laws beyond this. The arithmetic of realization spaces can also be arbitrarily complicated. Using enough points, one may construct an incidence variety that tests for roots in the base field $k$ of any given polynomial \cite[Section 2.1]{MR1009366}.  One can also detect characteristic; for example, $\mc{R}(S_{\Fano})$ is empty unless the characteristic of the base field is $2$.

Despite the universal nature of incidence varieties, there are many interesting questions to pursue regarding their structure, several of which we collect below.

\begin{question}[more points]
    
In principle, it should be possible to study incidence varieties constructed with $n$ points for larger values of $n$ by reducing to incidence varieties of superfigurations. But there is a combinatorial explosion, as shown in Table \ref{tab:linearspaces}.
\begin{center}
\small
\begin{table}[!ht]
\begin{tabular}{|c|r|r|}
\hline
    $n$ & \# linear spaces & \# $n_3$-superfigurations \\ \hline
    1 & 1 & 0 \\
    2 & 1 & 0 \\
    3 & 2 & 0 \\
    4 & 3 & 0 \\
    5 & 5 & 0 \\
    6 & 10 & 0 \\
    7 & 24 & 1 \\
    8 & 69 & 1 \\
    9 & 384 & 10 \\
    10 & 5250 & 151 \\
    11 & 232929 & 16234 \\
    12 & 28872973 & $>$179000 \\
    \hline
\end{tabular}
\label{tab:linearspaces}
\caption{The number of linear spaces and $n_3$-superfigurations on $n$ points, for $1 \leq n \leq 12$ \cite{MR1670277, MR3721583}.}
\end{table}
\end{center}

It might be possible to analyze all $16234$ realization spaces of $11_3$-superfigurations by computer. We expect that the majority of these spaces are rational. 

Rather than continuing in pursuit of complete classification, we think it would be most interesting to know how to realize specific kinds of components (e.g. varieties of general type, non-reduced schemes, singular varieties) using a minimal number of points. The only result we know of in this direction is Sturmfels' version of the universality theorem, which shows that the number of points required to realize a $\Q$-variety $V$ as an incidence variety component (up to birational transformations) is a polynomial in the degrees and coefficients of the equations of $V$ \cite{MR1002208}. 

Alternatively, one could restrict to interesting subclasses of linear spaces that might have realization spaces with additional structure. For example, a heuristic argument suggests that $n_3$-configurations should have Calabi-Yau realization spaces \cite{sink}. What about linear spaces with extremal numbers of lines, self-dual linear spaces, and linear spaces with large automorphism groups?
\end{question}

\begin{question}[unexpected dimensions]
    While the $10_3$-configurations all have the same number of points and lines, their realization spaces may be empty or nonempty, and the nonempty realization spaces may have dimension $2$ or $3$, indicating that there is no easy way to predict even the dimension of a realization space. This is closely related to the geometric configuration theorems of Pappus and Desargues that force the existence of unexpected lines. What is the maximum possible dimension of a realization space in terms of $n$? At the other extreme, what are the asymptotics of the proportion of unrealizable linear spaces as $n \to \infty$?
\end{question}

\begin{question}[realizability over $\R$]
Realizability over $\R$ is a central topic in the study of $n_3$-configurations \cite{MR2510707}. When a configuration is realizable over $\R$, it can be drawn with pen and paper using true lines, as opposed to the bent lines that appear in Figure \ref{fig_fano}. Using the methods of this paper, it should be possible to solve the following problems.
\begin{enumerate}
    \item Which linear spaces on at most $10$ points have realizations over $\R$?
    \item For which linear spaces are the realizations over $\Q$ dense in the realizations over $\R$? Sink addressed this problem for the $10_3$-configurations \cite{sink}.
\end{enumerate}
\end{question}

\begin{question}[finite characteristic]
Some superfigurations are unrealizable over $\Q$ but are nevertheless realizable in some finite characteristics. This behavior can be detected by working with incidence schemes over $\Spec \Z$. Given a prime $p$, how many points are needed to define an incidence scheme that is empty over $\Q$ but has points over $\F_p$?
\end{question}

\begin{question}[higher dimension]
The incidence varieties in $\PP^3$ determined by up to $7$ points with prescribed lines and planes were classified in \cite{MR4379254}. The methods therein can be modified to show that all such incidence varieties are rational. What is the set of incidence varieties determined by up to $8$ points in $\PP^3$?
\end{question}

\begin{question}[enumerative problems]
Our original motivation for classifying incidence varieties was towards counting $n$-arcs over finite fields, i.e., $n$-tuples of points in $\PP^2(\F_q)$, no $3$ of which are collinear. The counting functions for $n$-arcs for all $n \leq 9$ are quasipolynomial in $q$ \cite{MR1670277, MR957209, MR1391025}. The results of this paper lead us to expect that the counting function for $10$-arcs is non-quasipolynomial, because this counting function should involve the number of $\F_q$-points on certain K3 surfaces and elliptic curves. 
 However, the results of this paper do not quite give enough information to describe the $10$-arc counting function, since we only classify incidence varieties up to birational equivalence.
\end{question}

\section{Appendix}

In this appendix, we give a careful proof of Proposition \ref{prop_ls_reduction}, which we used to reduce our study from general incidence varieties to incidence varieties of linear spaces. We show that every component of an incidence variety $\mc{I}(n, L)$ is isomorphic to a component of an incidence variety over $k$ of a linear space on at most $n$ points.

It is helpful to enlarge the set of objects under consideration by viewing linear spaces as special cases of a certain kind of function, following \cite{MR1391025} (see also \cite{MR3721583}).

\begin{defn} \label{def_bool}
Given $n \in \N$, let $[n] \colonequals \{1,2,\ldots, n\}$.
A \textit{boolean n-function} is a function taking subsets of $[n]$ to $\{0,1\}$. Boolean $n$-functions $f$ and $g$ are \textit{isomorphic} if there is a permutation $i$ of $[n]$ such that $g = f \circ i$.
\end{defn}

In our applications, we will be interested in sets $L$ of subsets of $[n]$. The indicator function $1_L$ is a boolean $n$-function. In the other direction, to any boolean $n$-function $f$, we may associate $f^{-1}(1)$, and these operations are inverse. We therefore may identify boolean $n$-functions with sets of subsets of $[n]$.

\begin{defn}
Given boolean $n$-functions $f$ and $g$, we say $f \geq g$ if $f(I) \geq g(I)$ for all $I \subseteq \{1,2,\hdots,n\}$.
\end{defn}

The relation $\geq$ defines a partial order on the set of boolean $n$-functions, for each $n \in \N$. Our convention is to only use $\geq$ to compare between functions with the same value of $n$.

\begin{defn}
Every linear space on $[n]$ has an associated boolean $n$-function, namely, the indicator function for collinear subsets. A \emph{linear space function} is any boolean $n$-function that arises in this way. We sometimes highlight the value $n$ by saying $f$ is \textit{on $n$ points.}

The set of linear space functions on $n$ points is denoted $\LSF_n$. It is an embedded copy of the poset $\LS_n$ in the poset of boolean $n$-functions. 
\end{defn}

\begin{lemma}
Each linear space function is the associated function of a unique linear space. Further,
a boolean $n$-function $f$ is in $\LSF_n$ if and only if it satisfies the following:
\begin{enumerate}[ label={(\arabic*)} ]
\item If $f(I) = 1,$ then for all $J \subseteq I$, we have $f(J) = 1$.
\item If $\#I \leq 2,$ then $f(I) = 1$.
\item If $f(I) = f(J) = 1$ and $\# (I \cap J) \geq 2$, then $f(I \cup J) = 1$.
\end{enumerate}
\end{lemma}

\begin{proof}
Every linear space function defines a linear space on $[n]$ by defining the lines to be the maximal elements of $f^{-1}(1)$ of cardinality greater than $1$, with respect to inclusion. This operation is easily seen to be inverse to the construction of the associated linear space function. The remaining claims (1)-(3) follow immediately from the definitions.
\end{proof}

Henceforth, we identify linear space functions with the linear spaces they define. Our next step is to define a linear space function that is, in some sense, the best possible approximation to a given boolean $n$-function. 
This should be the minimal linear space function that contains at least the collinear subsets prescribed by $f$.

\begin{defn} \label{def_ls_ification}
Given a boolean $n$-function $f$, let $\bar{f}$ be the minimal linear space function $f'$ with respect to $\leq$ such that $f \leq f'$. (The existence and uniqueness of $\bar{f}$ are justified in Lemma \ref{lem_ls_ification} below.) 

Let $L$ be a set of subsets of $[n]$. We denote by $\bar{L}$ the line set of the linear space defined by $\bar{f}$, where $f$ is the boolean $n$-function $1_L$.
\end{defn}

\begin{lemma} \label{lem_ls_ification}
In Definition \ref{def_ls_ification}, the linear space function $\bar{f}$ exists and is unique.
\end{lemma}

\begin{proof}

Let $f$ be a boolean $n$-function. The constant function $1$ satisfies $f \leq 1$, so the set $\{h \in \LSF_n : h \geq f\}$ is non-empty. We claim that
$$f' = \bigwedge \{h \in \LSF_n : h \geq f\}$$
is a linear space function, where $\wedge$ denotes pointwise $\min$. This easily implies that $\ovl{f} = f'$.

We introduce three ``partial closure'' operations. Given a boolean $n$-function $f$:
\begin{enumerate}
    \item Let $f^{(1)}$ be the boolean $n$-function 
    \[
    f^{(1)}(I) = 
    \begin{cases}
        1 & \text{if $f(I) = 1$ or $\#I \leq 2$},\\
        0 & \text{otherwise}.
    \end{cases}
    \]
    \item Let $f^{(2)}$ be the boolean $n$-function 
    \[
    f^{(2)}(I) = 
    \begin{cases}
        1 & \text{if $f(I') = 1$ for some $I' \supseteq I$,}\\
        0 & \text{otherwise}.
    \end{cases}
    \]
    \item Let $f^{(3)}$ be the boolean $n$-function
    \[
    f^{(3)}(I) = 
    \begin{cases}
        1 & \text{if $f(I) = 1$, or } \\ 
        & \text{\quad ($I = I' \cup I''$ for some $I', I'' \in f^{-1}(1)$ such that $\#(I' \cap I'') \geq 2$),}\\
        0 & \text{otherwise}.
    \end{cases}
    \]
\end{enumerate}
Notice that $f \leq f^{(1)}, f \leq f^{(2)}, f \leq f^{(3)}$. Notice further that $f \in \LSF_n$ if and only if all four of these functions are equal. Since the set of boolean $n$-functions is finite and its maximal element is a linear space function, it follows that some linear space $g$ may be obtained from $f$ by repeatedly applying these three operations in some order. Notice also that these three operations $(1), (2), (3)$ respect $\leq$. So if $h \in \LSF_n$ and $h \geq f$, then $h \geq g$. So $g = f'$.
\end{proof}

Now we are almost prepared to prove the main result of this appendix, which concerns components of incidence varieties. While one might guess that $\mc{I}(n, L) = \mc{I}(n, \bar{L})$, this is not true in general, as illustrated by Example \ref{ex_ls_reduction}. However, the two incidence varieties become equal after restricting attention to \emph{injective} set maps. To see this, we make use of the following universal property.

\begin{prop} \label{prop_closure}
    Let $L$ be a set of subsets of $[n]$. Then $(n, \ovl{L})$ is the unique linear space $S$ with the following property. For all linear spaces $T = (P', L')$ and injective set maps $\rho: [n] \hookrightarrow P'$, the map $\rho$ is a map of linear spaces $S \hookrightarrow T$ if and only if, for every $\ell \in L$, $\{\rho(i) : i \in \ell\}$ is collinear.
\end{prop}

\begin{proof}
    Assuming the existence of a linear space satisfying the condition of the proposition, we prove that it is unique. If $S$ and $S'$ are linear spaces with the desired property, then the identity map $\rho: [n] \to [n]$ is a map of linear spaces $S \to S'$ and $S' \to S$, so $S \leq S'$ and $S' \leq S$, which yields $S = S'$.

    We now prove the existence of a linear space satisfying the condition of the proposition. Let $f=1_L$, so $L = f^{-1}(1)$. Then $\bar{f}$ defines a linear space $S = (n,\bar{L})$, and we claim $S$ has the desired property.  Suppose that $T = (P', L')$ is a linear space and $\rho: [n] \hookrightarrow P'$ is injective.

    $\Rightarrow:$ If $\rho$ is a linear space map $S \to T$, then it sends collinear subsets in $S$ to collinear subsets in $T$. Each element of $\bar{f}^{-1}(1)$ is collinear in $S$. Then notice that $\bar{f}^{-1}(1) \supseteq f^{-1}(1)$.

    $\Leftarrow:$ For every $\ell \in f^{-1}(1)$, we have that $\{\rho(i) : i \in \ell\}$ is collinear, by assumption.
    In the notation of the proof of Lemma~\ref{lem_ls_ification},
    we claim that for every $\ell$ sent to $1$ by $f^{(1)}$, $f^{(2)}$, or ${f^{(3)}}$, we have that $\{\rho(i)\colon i \in \ell\}$ is collinear. This suffices, since repeatedly applying the partial closure operations yields that every $\ell \in \bar{f}^{-1}(1)$ is sent by $\rho$ to a collinear set. 
    \begin{itemize}[leftmargin=*]
    \item To see (1), suppose that $f^{(1)}(I) = 1$. Then $f(I)=1$ or $\#I = 1$. In the first case, by assumption $\{\rho(i) : i \in \ell\}$ is collinear. In the second case $\{\rho(i) : i \in \ell\}$ is trivially collinear.
    \item To see $(2)$, suppose that $f^{(2)}(I) = 1$. Then $f(I)=1$, in which case $\{\rho(i) : i \in I\}$ is collinear by assumption, or $f(I') = 1$ for some $I' \supseteq I$. In the latter case $\{\rho(i)\colon i \in I'\}$ is collinear by assumption, and since $T$ is a linear space, the subset $\{\rho(i)\colon i \in I\}$ is also collinear.
    \item To see $(3)$, suppose that $f^{(2)}(I) = 1$. Then $f(I)=1$, in which case $\{\rho(i) : i \in I\}$ is collinear by assumption, or $I = I' \cup I''$ where $f(I')=1$ and $f(I'')=1$ and $I' \cap I''$ has cardinality at least $2$. By the injectivity of $\rho$, the image $\rho(I' \cap I'')$ is not trivially collinear, so it is contained in a line $\ell \in L'$. By the definition of a linear space, no other line of $L'$ contains $I' \cap I''$. Since $\{\rho(i)\colon i \in I'\}$ and $\{\rho(i)\colon i \in I''\}$ are collinear and each intersect $\ell$ in at least $2$ points, their union $\{\rho(i)\colon i \in I\}$ is contained in $\ell$.
    \end{itemize}
\end{proof}

In the proof of Proposition 2.4, we consider parametrizations of injective maps into $\PP^2$ and non-injective maps to $\PP^2$ separately.  To handle the non-injective maps, we need the following definitions. Suppose $f$ is a boolean $n$-function and $\rho:[n] \twoheadrightarrow [m]$. The \emph{quotient function} $g$ is the boolean $m$-function defined by $g^{-1}(1) = \{ \rho(\ell) : \ell \in f^{-1}(1) \}$. Notice that the quotient of a linear space function need not define a linear space. Suppose $S = (P, L)$, and let $\rho: P \twoheadrightarrow Q$. The \emph{quotient space} is the linear space on $Q$ defined by $\bar{f}$, where after relabeling, $f^{-1}(1) = \{\rho(\ell) : \ell \in L\}$.

\begin{prop}[= Proposition \ref{prop_ls_reduction}]
    Let $L$ be a set of subsets of $[n]$. Every component of the incidence variety $\mc{I}(n, L)$ is isomorphic to a component of an incidence variety over $k$ of a linear space on $\leq n$ points.
\end{prop}

\begin{proof}
We argue by induction on $n$.  The case $n=1$ is clear because the choices for $L$ are $\{\}, \{ \varnothing \}, \{\{1\}\}, \{ \varnothing,\{1\}\}$, these all give rise to the same incidence variety $\PP^2$, and the first choice $(1, \varnothing)$ is a linear space. 

Suppose the claim holds for $n-1$. The incidence variety $\mc{I}(n,L)$ parametrizes maps $[n] \to \PP^2$ such that the image of every set in $L$ is collinear. Taking $T = \PP^2$ in Proposition \ref{prop_closure}, an injective set map $\rho : [n] \hookrightarrow \PP^2$ is a weak realization of $(n,\bar{L})$ and only if the point in $(\PP^2)^n$ representing $\rho$ lies in $\mc{I}(n,L)$. Recall that we defined $\Delta = \bigcup_{i \neq j} \Delta_{ij}$, where $\Delta_{ij} = \{p_i = p_j\}$. Since injective set maps are precisely those represented in the complement of $\Delta$,
we have
$$\mc{I}(n,L) \smallsetminus \Delta = \mc{I}(n, \bar{L}) \smallsetminus \Delta.$$
Because incidence varieties are closed, any component of $\mc{I}(n,L)$ not contained in $\Delta$ agrees with a component of $\mc{I}(n, \bar{L})$.

Now, to prove the proposition, we only must show that components of $\mc{I}(n,L)$ contained in $\Delta$ are each isomorphic to components of incidence varieties of linear spaces on at most $n$ points. Maps $[n] \to \PP^2$ with image in $\Delta_{ij}$ are in one-to-one correspondence with maps $Q \to \PP^2$, where $Q$ is the quotient set $[n]/\sim$ obtained by identifying $i,j$. So $\mc{I}(n,L) \cap \Delta_{ij} \cong \mc{I}(n-1,L')$, where $L'$ is the quotient of $L$ after identifying $i,j$ (and labeling the identified points in any way). The inductive hypothesis then shows that all components of this form come from incidence varieties of linear spaces on $n - 1$ points. Since every component contained in $\Delta$ is contained in some $\Delta_{ij}$, we are done.
\end{proof}

\bibliographystyle{plain}
\bibliography{bib}

\end{document}